\documentclass[11pt]{amsart}
\usepackage{pgf,tikz,pgfplots,color}
\pgfplotsset{compat=1.15}
\usepackage{mathrsfs} 
\usetikzlibrary{arrows}
\usepackage{fancyhdr}
\usepackage{lastpage} 

\usepackage[margin=1in]{geometry}
\usepackage{geometry}
\usepackage[toc,page]{appendix}
\usepackage[utf8]{inputenc}
\usepackage{hyperref}
\hypersetup{
	colorlinks=true,
	linkcolor=blue,
	citecolor=brown,
	filecolor=magenta, 
	urlcolor=blue,
}
\usepackage[notcite,notref]{}   
\usepackage{amsmath}
\usepackage{amsfonts}
\usepackage{amssymb}
\usepackage{amsthm}
\usepackage{setspace}
\usepackage{inputenc}
\usepackage[english]{babel} 
\usepackage{comment}
\usepackage[shortlabels]{enumitem}
\numberwithin{equation}{section}

\makeatletter
\def\@tocline#1#2#3#4#5#6#7{\relax
  \ifnum #1>\c@tocdepth 
  \else
    \par \addpenalty\@secpenalty\addvspace{#2}%
    \begingroup \hyphenpenalty\@M
    \@ifempty{#4}{%
      \@tempdima\csname r@tocindent\number#1\endcsname\relax
    }{%
      \@tempdima#4\relax
    }%
    \parindent\z@ \leftskip#3\relax \advance\leftskip\@tempdima\relax
    \rightskip\@pnumwidth plus4em \parfillskip-\@pnumwidth
    #5\leavevmode\hskip-\@tempdima
      \ifcase #1
       \or\or \hskip 1em \or \hskip 2em \else \hskip 3em \fi%
      #6\nobreak\relax
    \hfill\hbox to\@pnumwidth{\@tocpagenum{#7}}\par
    \nobreak
    \endgroup
  \fi} 
\makeatother

\title[]{On $1/2$ estimate for global Newlander-Nirenberg theorem}           

\author[]{Ziming Shi}

\address{Department of Mathematics,
	Southern University of Science and Technology, ShenZhen, China}
\email{shizm@sustech.edu.cn}


\newcommand{\dist}{\operatorname{dist}}

\newcommand{\supp}{\operatorname{supp}}
\newcommand{\Span}{\operatorname{span}}

\newcommand{\lip}{\operatorname{Lip}}

\newtheorem{thm}{Theorem}[section]
\newtheorem{cor}[thm]{Corollary} 
\newtheorem{prop}[thm]{Proposition}
\newtheorem{lemma}[thm]{Lemma}

\newcommand{\fl}[1]{\lfloor #1\rfloor}

\theoremstyle{definition}
\newtheorem{defn}[thm]{Definition}
\newtheorem{exmp}[thm]{Example}
\newtheorem{ques}[thm]{Question}
\newtheorem{note}[thm]{Notation}

\theoremstyle{remark}
\newtheorem{rem}[thm]{Remark}
\newtheorem*{clm}{Claim}
\newtheorem*{ack}{Acknowledgment}

\renewcommand{\th}[1]{\begin{thm}\label{#1}}
	\renewcommand{\eth}{\end{thm}}
\newcommand{\co}[1]{\begin{cor}\label{#1}}
	\newcommand{\eco}{\end{cor}}
\newcommand{\pr}[1]{\begin{prop}\label{#1}}
	\newcommand{\epr}{\end{prop}}

\newcommand{\df}[1]{\begin{defn}\label{#1}}
	\newcommand{\edf}{\end{defn}}
\newcommand{\ex}[1]{\begin{exmp}\label{#1}} 
	\newcommand{\eex}{\end{exmp}}
\newcommand{\qu}[1]{\begin{ques}\label{#1}}
	\newcommand{\equ}{\end{ques}}  
\newcommand{\mk}{\begin{rem}}
	\newcommand{\emk}{\end{rem}}
\newcommand{\cl}{\begin{clm}}
	\newcommand{\ecl}{\end{clm}} 
\newcommand{\ac}{\begin{ack}}
	\newcommand{\eac}{\end{ack}} 

\newcommand{\ga}{\begin{gather}}
\newcommand{\ega}{\end{gather}}
\newcommand{\gan}{\begin{gather*}}
\newcommand{\egan}{\end{gather*}}
\newcommand{\al}{\begin{gngn}}
	\newcommand{\eal}{\end{align}}
\newcommand{\aln}{\begin{align*}}
\newcommand{\ealn}{\end{align*}}
\newcommand{\eq}[1]{\begin{equation}\label{#1}}
\newcommand{\eeq}{\end{equation}}

\newcommand{\pa}{\partial{}}
\newcommand{\na}{\nabla}
\newcommand{\db}{\dbar}

\newcommand{\we}{\wedge}

\newcommand{\sm}{\setminus}
\newcommand{\seq}{\subseteq}

\newcommand{\pp}[2]{\frac{\partial #1}{\partial #2}}

\newcommand{\Z}{\mathbb{Z}}
\newcommand{\R}{\mathbb{R}} 
\newcommand{\C}{\mathbb{C}}

\newcommand{\N}{\mathbb{N}}

\newcommand{\U}{\mathcal{U}} 

\newcommand{\mf}{\mathfrak}

\newcommand{\0}{\mathbf{0}} 
\newcommand{\1}{\mathbf{1}}

\newcommand{\ov}{\overline}
\newcommand{\ti}{\tilde}
\newcommand{\wti}{\widetilde}
\newcommand{\hht}{\widehat}

\newcommand{\dbar}{\overline\partial}

\newcommand{\all}{\alpha}

\newcommand{\del}{\delta}

\newcommand{\var}{\varphi}
\newcommand{\e}{\epsilon}
\newcommand{\ve}{\varepsilon}
\newcommand{\om}{\omega}
\newcommand{\Om}{\Omega}
\newcommand{\thh}{\theta}
 
\newcommand{\La}{\Lambda}
\newcommand{\la}{\lambda}

\newcommand{\gm}{\gamma}

\newcommand{\si}{\sigma}

\newcommand{\yh}{\frac{1}{2}}

\newcommand{\re}[1]{(\ref{#1})}

\newcommand{\rl}[1]{Lemma~\ref{#1}}
\newcommand{\rc}[1]{Corollary~\ref{#1}}
\newcommand{\rp}[1]{Proposition~\ref{#1}}
\newcommand{\rt}[1]{Theorem~\ref{#1}}
\newcommand{\rd}[1]{Definition~\ref{#1}}

\newcommand{\nn}{\nonumber}
\newcommand{\nid}{\noindent}

\newcounter{pp}
\newcommand{\bpp}{\begin{list}{$\hspace{-1em}\alph{pp})$}{\usecounter{pp}}}
	\newcommand{\epp}{\end{list}}

\newcounter{ppp}
\newcommand{\bppp}{\begin{list}{$\hspace{-1em}(\roman{ppp})$}{\usecounter{ppp}}}
	\newcommand{\eppp}{\end{list}}

\newcommand{\Ac}{\mathcal{A}}

\newcommand{\Dc}{\mathcal{D}}

\newcommand{\Bs}{\mathscr{B}}
\newcommand{\Cf}{\mathfrak{C}}
\newcommand{\Df}{\mathfrak{D}}
\newcommand{\Ec}{\mathcal{E}}

\newcommand{\Gf}{\mathfrak{G}}
\newcommand{\Hc}{\mathcal{H}}
\newcommand{\Kb}{\mathbb{K}}
\newcommand{\Kc}{\mathcal{K}}

\newcommand{\Sc}{\mathcal{S}}
\newcommand{\Ss}{\mathscr{S}}
\newcommand{\Tc}{\mathcal{T}}

\newcommand{\Uc}{\mathcal{U}}

\begin{document} 
	\definecolor{rvwvcq}{rgb}{0.08235294117647059,0.396078431372549,0.7529411764705882}
\maketitle  

\begin{abstract} 
  Given a formally integrable almost complex structure $J$ defined on the closure of a bounded domain $D \subset \mathbb C^n$, and provided that $J$ is sufficiently close to the standard complex structure, the global Newlander-Nirenberg problem asks whether there exists a global diffeomorphism defined on $\overline D$ that transforms $J$ into the standard complex structure, under certain geometric and regularity assumptions on $D$. In this paper we prove a quantitative result of this problem. Assuming $D$ is a strictly pseudoconvex domain in $\mathbb C^n$ with $C^2$ boundary, and that the almost complex structure $J$ belongs to the H\"older-Zygmund class $\Lambda^r(\overline D)$ for $r>\frac{3}{2}$, we show the existence of a global diffeomorphism (independent of $r$) in the class $\Lambda^{r+\frac12-\varepsilon}(\overline D)$, for any $\varepsilon>0$. 
\end{abstract} 
\vspace{1cm} 
	
\tableofcontents

\medskip   
\section{Introduction}
Let $M$ be a real differentiable manifold. An almost complex structure near a point $p$ in $M$ is defined as a tensor field $J$, which is, at every point in some neighborhood $U$ of $p$, an endomorphism of the tangent space $T_x(M)$ such that $J^2 = -I$, where $I$ denotes the identity transformation of $T_x(M)$. Consequently there exists a decomposition of the complexified tangent bundle into the $+i$ and $-i$ eigenspaces of $J$: $\C TM = \Sc^+_J \oplus \Sc^-_J$. If $M$ is a complex manifold with local holomorphic coordinate chart $\{ U^k, z^k \}$, then the complex structure on $M$ gives rise to the standard almost complex structure $J_{st}$, defined by $S^+_{J_{st}} = \Span 
\{ \pp{}{z^k_\all} \}_{\all=1}^n $ on $U^k$. 
Conversely, given an almost complex structure $J$ on $M$, one wants to know whether $J$ is induced by the complex structure on $M$, in other words, whether there exists a diffeomorphism $z_\all \to w_\beta$ so that in the new coordinate $S^+_J= \Span \left\{ \pp{}{w_\all} 
\right\}_{\all=1}^n$. In this case we say that $J$ is integrable, and $\{ w_\all  \}_{\all=1}^n$ is a holomorphic coordinate chart compatible with $J$. 

We call an almost complex structure $J$ formally integrable if $S^+_J$ is closed under the Lie bracket $[\cdot, \cdot] $ (An equivalent condition is the vanishing of the Nijenhuis tensor.) In particular, an integrable structure is formally integrable. 
The classical local Newlander-Nirenberg theorem asserts that the converse is also true locally, i.e. for a formally integrable almost complex structure $J$ defined near an interior point of $M$, there exists a local holomorphic coordinate system in a neighborhood of $p$ which is compatible with $J$.  

For real analytic $J$, the local Newlander-Nirenberg theorem follows easily from the analytic Frobenius theorem. However if $J$ is only $C^\infty$ or less regular, the proof becomes much more difficult. The reader may refer to the work of Newlander-Nirenberg \cite{N-N57}, Nijenhuis-Woolf \cite{N-W63}, Malgrange \cite{Mal69} and Webster \cite{We89}. We point out that Webster's proof yields the sharp regularity result in the H\"older space, namely, if $J$ is $C^{k+\all} $ for $k \geq 1$ and $\all \in (0,1)$ in a neighborhood of a point $p$,   then there exists a local diffeomorphism near $p$ 
of the class $C^{k+1+\all}$ such that the new coordinate is compatible with $J$.   

We now consider the analogous global problem. Suppose a formally integrable almost complex structure $J$ is defined on the closure of a relatively compact subset $D$ in a complex manifold $M$, such that $J$ is a small perturbation of $J_{st}$ (as measured by certain norms such as the H\"older or H\"older-Zygmund norm), one wants to know whether $J$ is integrable on $\ov D$, or equivalently, if there exists a global diffeomorphism on $\ov D$ inducing a holomorphic coordinate system compatible with $J$. Moreover, we are interested in the global regularity of such coordinate. 
We shall henceforth refer to this problem as the global Newlander-Nirenberg problem.  
 
Under the assumption that both the boundary $bD$ and the almost complex structure $J$ are $C^\infty$, Hamilton \cite{Ha77} proved the existence of a $C^\infty$ diffeomorphism on $\ov D$ under which the new coordinate is compatible with $J$, if $D$ satisfies 
1) $H^1(D, \Tc D) = 0$, where $\Tc$ stands for the holomorphic tangent bundle of $D$; and 2) the Levi form on $bD$ has either $n-1$ positive eigenvalues or at least two negative eigenvalues. There is also a local version of Newlander-Nirenberg theorem with boundary for strictly pseudoconvex hypersurface, due to Catlin \cite{Ca88} and Hanges-Jacobowitz \cite{H-J89}, independently. We note that all these results are carried out in the $C^\infty$ category (boundary, structure and the resulting diffeomoprhism are all $C^\infty$) using $\dbar$-Neumann-type methods. 
More recently, Gan and Gong \cite{G-G24} proved a global Newlander-Nirenberg theorem on a strictly pseudoconvex domain $D$ in $\C^n$ with $C^2$ boundary. Assuming $J \in \La^r(\ov D)$, $r>5$, they proved the existence of a diffeomorphism in the class $\La^{r-1}(\ov D)$, assuming that $|J-J_{st}|_{\La^r(\ov D)} < \del(r)$ for some sufficiently small $\del(r)$. Their method is based on the $\db$ homotopy formula, an approach originally pioneered by Webster in his proof of the classical (local) Newlander-Nirenberg theorem \cite{We89}.   

To formulate our results we first introduce some notations. Let $p \in M$, we identify an almost complex structure $J$ near $p$ by a set of vector fields $\{ X_{\ov \all} \}_{\all=1}^n$ such that
\[
  \text{ $X_{\ov 1}, \dots, X_{\ov n}, \ov{X_{\ov 1}}, \dots, \ov{X_{\ov n}}$ are linearly independent at $p$},  
\]
where $\{ X_{\ov \all} \}$ spans the eigenspace $S_J^-$ with eigenvalue $-i$. 
We can write  
\[
  X_{\ov \all} = a_{\ov \all}^\beta \pp{}{z_\beta} + b^\beta_{\ov \all} \pp{}{\ov z_\beta} .  
\]
Here we have used Einstein convention to sum over the repeated index $\beta$, and we shall adopt this convention throughout the paper. 
$\ov {X_{\ov \all}}$ denotes the complex conjugate of $X_{\ov \all}$ and $\{ \ov{X_{\ov \all} } \}_{\all=1}^n$ is a basis for the eigenspace $S_J^+$ with eigenvalue $i$.  

For two vector fields $V,W$ on $M$, we denote their Lie bracket as $[V, W] = VW-WV$. The formal integrability of $J$ on $\ov D$ means there exist functions $C^{\ov \gm}_{\ov \all \ov \beta}$ such that 
\[
  [X_{\ov \all}, X_{\ov \beta} ] = C^{\ov \gm}_{\ov \all \ov \beta} X_{\ov \gm}, \quad \text{for $\all, \beta= 1, \dots, n$} 
\]
everywhere on $\ov D$. 

By an $\R$ linear change of coordinates, we can transform $X_{\ov \all}$ to the form 
\[
  X_{\ov \all} = \pp{}{\ov z_\all} + A_{\ov \all}^\beta \pp{}{z_\beta}. 
\]
We now state our main result for the global Newlander-Nirenberg problem. We use the notation $a^-$ to mean that $a-\ve$ for any $\ve >0$. We use $\La^r(\ov D)$ or $|\cdot|_{D, r}$ to denote the H\"older-Zygmund norm on $D$ (See \rd{Def::H-Z} below).   
\begin{thm}  \label{Thm::intro_main_thm}  
Let $D$ be a domain in $\C^n$ with $C^2$ boundary that is strictly pseudoconvex with respect to the standard complex structure in $\C^n$, for $n \geq 2$. Let $3/2 < m \leq \infty$ and let $\{ X_{\ov \all} = \pp{}{\ov z_\all} + A_{\ov \all}^\beta \pp{}{z_\beta} \}_{\all=1}^n$ be $\La^m(\ov D)$ vector fields defining 
a formally integrable almost complex structure on $\ov D$. Let $\e_0, \ti \e_0$ be small positive constants such that $m > \frac{3}{2} + \ti \e_0$ and $0 < \e_0 < \ti \e_0$. 
There exists $\del_0= \del_0(D, |A|_{\frac32+ \ti \e_0}, \ti \e_0)>0$ such that if $|A|_{D,1+\e_0}  < \del_0$, then there exists an embedding $F$ of $\ov D$ into $\C^n$ such that $dF(X_{\ov 1}), \dots, dF(X_{\ov n})$ are in the span of $\left\{ \pp{}{\ov z_1}, \dots, \pp{}{\ov z_n} \right\} $.  
Furthermore, $F \in \La^{m+\yh^-}(\ov D)$ if $m < \infty$ and $F \in C^\infty(\ov D)$ if $m = \infty$. The constant $\del_0$ needs to converge to $0$ as $m \to \frac{3}{2}^+$ ($\ti \e_0 \to 0$) and can be chosen to be independent of all $m$ away from $\frac32$.  
\end{thm} 

\rt{Thm::intro_main_thm} proves the almost $1/2$ gain in regularity for the global Newlander-Nirenberg problem on strictly pseudoconvex domains in $\C^n$. It is worthwhile to note that our result is achieved assuming only that the initial almost complex structure is a small perturbation in the $\La^{1+\e_0}(\ov D)$ norm from the standard structure $J_{st}$. This is a major improvement over the previous best known result in \cite{G-G24}, where one needs to assume the smallness of the  perturbation in the $\La^m(\ov D)$ norm in order to obtain a diffeomorphism in $\La^{m-1}(\ov D)$, for $m >5$.  

The constant $\del_0$ is lower stable under small $C^2$ perturbation of the domain (see Definition \ref{Def::upp_low_stable}.) As a consequence, we can also prove  the following local Newlander-Nirenberg theorem with boundary, improving the results of \cite{Ca88}, \cite{H-J89} and \cite{G-G24}. 
\begin{thm} \label{Thm::Loc_N-N} 
  Let $3/2 < m \leq \infty$. Let $U$ be a domain in $\C^n$ whose boundary contains a piece of $C^2$ strictly pseudoconvex real hypersurface $M$, and let $X_{\ov 1}, \dots, X_{\ov n} \in \La^m(U \cup M)$ vector fields defining a formally integrable almost complex structure on $U \cup M$. Then for each $p \in M$, there exists a diffeomorphism $F$ defined on a neighborhood $\om$ of $p$ in $U \cup M$ such that $dF(X_{\ov 1}), \dots, dF(X_{\ov n})$ are in the span of $\left\{ \pp{}{\ov z_1}, \dots, \pp{}{\ov z_n} \right\} $, and $F(\om \cup M)$ is strictly pseudoconvex. Furthermore, $F \in \La^{m+\yh^-}(\ov \om)$ if $m <\infty$ and $F \in C^\infty(\ov \om)$ if $m = \infty$. 
\end{thm}

We now describe our method, which is a modified form of the KAM type argument of \cite{We89} and \cite{G-G24}. Given $D=D_0$ with an initial almost complex structure $J_0$, we look for a succession of diffeomorphisms $\{F_i \}_{i=1}^\infty$ that maps $D_{i}$ with structure $J_i$ to a new domain $D_{i+1}$ with structure $J_{i+1}$. During the iteration, the perturbation $\|A_i\|_{C^0(\ov D_i)} = \|J_i-J_{st}\|_{C^0(\ov D_i)}$ converges to $0$ while each $D_i$ remains a strictly pseudoconvex domain with $C^2$ boundary. 
The sequence of the domains $D_i$ converges to a limiting domain $D_\infty$, while the sequence of diffeomorphisms $\wti F_j = F_j \circ F_{j-1} \cdots \circ F_1$ converges to a diffeomorphism $F:D_0 \to D_\infty$ with $dF (J_0) = J_{st}$. 
For each $i$, the map $F_i$ is constructed by applying a $\dbar$ homotopy formula $A_i = \dbar P_i A_i + Q_i \dbar A_i$ on $D_i$. In Webster's proof for the classical Newlander-Nirenberg theorem, only a local homotopy formula is needed, namely the Bochner-Martinelli-Koppelman formula. The operator $P_i$ gains one derivative, and as a result there is no loss of derivative for the almost complex structure at each iteration step. 
Using the integrability condition, Webster is able to obtain the rapid convergence of the perturbation for all derivatives: $|A_{i+1}|_{C^{k+\all}(\ov D)} \leq |A_i|^2_{C^{k+\all}(\ov D)}$ for any positive integer $k$ and $\all \in (0,1)$. 

For our problem, we need to apply a global homotopy formula on $\ov D$. The associated operators $P_i, Q_i$ gain only $1/2$ derivative up to boundary, which amounts to a loss of $1/2$ derivative for the new almost complex structure after each iterating diffeomorphism, and thus the iteration will break down after finite steps. 
To avoid the loss of derivatives, Gan and Gong \cite{G-G24} applied a Nash-Moser type smoothing operator 
at each step. In this case they show that the lower-order norms $|A_i|_{\La^s(\ov {D_i})}$ converges to $0$ for $s \in (2,3)$, while the higher-order norm $|A_i|_{\La^r(\ov {D_i})}$ blows up, for all sufficiently large $r$. By using the convexity of H\"older-Zygmund norms (interpolation), they can then show that the intermediate norms $|A_i|_{D_i,m}$ converges to $0$, for $ s < m <r$. In the iteration scheme of \cite{G-G24}, the higher-order norm blows up like $|A_{i+1}|_r \leq C_r t_i^{-\yh}|A_i|_r$, $r >5$, where $t_i$ is the parameter in the smoothing operator that tends to $0$ in the iteration. In our case, we modify their method by using a different diffeomorphism $F_i$  which allows us to prove the estimate 
\begin{equation} \label{intro_r_norm_est} 
  |A_{i+1}|_{\La^r(\ov D_{i+1})} \leq C_r |A_i|_{\La^r(\ov D_i)}, \quad r > 1.  
\end{equation} 
To achieve the above goal, we construct a family of smoothing operators $\{S_t\}_{t>0}$ acting on functions defined on a bounded Lipschitz domain and satisfying certain bounds in H\"older-Zygmund space, thus avoiding the use of the extension operator required for smoothing in \cite{G-G24}. For our construction we use Littlewood-Paley functions, which is a convenient tool since the H\"older-Zygmund norm $\La^r$ is equivalent to the Besov norm $\mathscr{B}^r_{\infty,\infty}$. 

The key feature in our estimate of $A_i$ is the nice property enjoyed by the commutator $[\na, S_t] = \na S_t - S_t \na$, which replaces the role of the commutator $[\na, E]$ in \cite{G-G24}, $E$ being some extension operator. 
Our estimate roughly states that if $u \in \La^r(\ov \Om)$, then for 
any $ 0< s <r$, the $\La^s(\ov \Om)$ norm of $[\na, S_t]$ tends to $0$ like $t^{r-s}$ (as $t \to 0)$.

It is plausible to conjecture that one should be able to gain exactly $1/2$ derivative in regularity for the global Newlander-Nirenberg problem, in view of the corresponding $1/2$ gain for the regularity of the $\dbar$ equation on strictly pseudoconvex domains. However, our method necessarily incurs a loss of arbitrarily small $\ve$ in smoothness, due to the use of the convex interpolation of norms. One could also ask whether the assumption that $J \in \La^r(\ov D)$, for $r>3/2$ can be replaced by 
$J \in C^1(\ov D)$ or $J \in \La^r(\ov D)$ for $r >1$. In our proof, the index $3/2$ comes from the need to control the $C^2$ norm and the Levi form of each iterating domain so that the domains remain strictly pseudoconvex.  

This non-linear, dynamical method of proving the Newlander-Nirenberg theorem using $\db$ homotopy formula 
was originated by Webster; it has found powerful applications in the CR vector bundle problem and the more difficult local CR embedding problem. See for example Gong-Webster \cite{G-W10,G-W11,G-W12} where they used such methods to obtain several sharp estimates on these problems. We also mention the paper by Polyakov \cite{Pol04} which uses similar techniques for the CR Embedding problems for compact regular pseudoconcave CR submanifold. 

The paper is organized as follows. In Section 2, we collect some basic properties of the H\"older Zygmund space $\La^s$. In particular we recall the Littlewood-Paley characterization of $\La^s$ using the Besov norm $\Bs^s_{\infty,\infty}$. In Section 2.2 we construct the important Moser-type smoothing operator on bounded Lipschitz domains. We also prove the commutator estimate for $[\dbar, S_t]$ in \rp{Prop::D_St_comm_est}, which plays a key role for estimating the perturbation in the iteration. 
In Section 3 we derive for each iteration the estimates for the lower and higher order norms of the new perturbation $A_{i+1}$ in terms of the norms of previous perturbation $A_i$. This constitutes the main estimates of the paper.  
In Section 4 we set up the iteration scheme using estimates from Section 3 and induction arguments, and we establish the convergence of the lower order norms and the blow up of higher order norms. Finally, we use convexity of norms to show that the composition of the iterating maps converges to a limiting diffeomorphism in suitable function spaces.

We now fix some notations used in the paper. We will often write $\pp{}{z_\all}$ as $\pa_{\all}$ and $\pp{}{\ov z_\beta}$ as $\pa_{\ov \beta}$, and we use $A \lesssim B$ to mean that there exists a constant $C$ independent of $A,B$ such that $A \leq CB $. For a map $f: \R^d \to \R^d$, we denote its Jacobian matrix by $Df = \left[ \pp{f^i}{x_j} \right]_{1 \leq i,j \leq d}$. The set of integers is denoted by $\Z$, and the set of natural numbers $\{0,1,2,3,\dots \}$ is denoted by $\N$.  

\begin{ack}
  The author would like to thank Liding Yao and Xianghong Gong for helpful discussions. 
\end{ack}

\bigskip  

\section{Preliminaries}  

\subsection{Function spaces}   
In this section, we recall some basic results for the H\"older space $C^r(\Om)$, $0 \leq r < \infty$, and the H\"older-Zygmund space $\La^r(\Om)$, $0 < r < \infty$.   

\begin{note}
   For simplicity we write the H\"older norm $\| \cdot \|_{C^r(\Om)}$ as $\| \cdot \|_{\Om,r}$ and $\| \cdot\|_{\La^r(\Om)}$ as $| \cdot |_{\Om,r}$. 
\end{note}
We now define the notion of special and bounded Lipschitz domains. 
\begin{defn} \label{Def::Lip_dom}  
    A \emph{special Lipschitz domain} is an open set $\om \subset \R^d$ of the form $\omega=\{(x',x_d):x_d>\rho(x')\}$ with $\|\nabla\rho\|_{L^\infty}< 1$. A \emph{bounded Lipschitz} domain is a bounded open set $\Om$ whose boundary is locally the graph of some Lipschitz function. In other words, $b \Om = \bigcup_{\nu=1}^M U_\nu$, where for each $1 \leq \nu \leq M $, there exists an invertible linear transformation $\Phi_\nu:\R^d \to \R^d$ and a special Lipschitz domain $\om_\nu$ such that 
    \[
      U_\nu \cap \Om = U_\nu \cap \Phi_\nu (\om_\nu). 
    \]
Fix such covering $\{ U_\nu \}$, we define the \emph{Lipschitz norm of $\Om$ with respect to $U_\nu$, denoted as $\lip_{U_\nu}(\Om)$}, to be $\sup_{\nu} \| D \Phi_\nu\|_{C^0}$.  
\end{defn}
For a bounded Lipschitz domain $\Om$, the following H\"older estimates for interpolation, product rule and chain rule are well-known. See for instance \cite{Ho76}. 
\begin{gather*}
  \| u \|_{\Om, (1-\thh)a + \thh b} \leq C_{a,b} \| u \|_{\Om,a}^{1-\thh} \| u \|_{\Om,b}^\thh, \quad 0 \leq \thh \leq 1; 
  \\  
\| u v \|_{\Om,a} \leq C_a ( \| u \|_{\Om,a} \| v \|_{\Om,0} + \| u \|_{\Om,0} \| v \|_{\Om,a} ) 
\\ 
\| u \circ \var \|_{\Om,a} \leq C_a \left( \| u \|_{\Om',a} \| \var \|_{\Om,1}^a + \| u \|_{\Om',1} \| \var \|_{\Om,a} + \| u \|_{\Om',0} \right). 
\end{gather*}
Here $a,b \geq 0$ and $\var: \Om \to \Om'$ is a $C^1$ map between two bounded Lipschitz domains $\Om,\Om'$. 
\begin{defn}[H\"older-Zygmund space] \label{Def::H-Z} 
The H\"older-Zygmund space on $\R^d$, denoted by $\La^s(\R^d)$ for $s\in\R^+$ is defined as follows
\begin{itemize}
    \item For $0<s<1$, $\La^s(\R^d)$ consists of all $f\in C^0(\R^d)$ such that $\|f\|_{\La^s(U)}:=\sup \limits_{\R^d} |f|+\sup\limits_{x,y\in \R^d, \, x \neq y} \frac{|f(x)-f(y)|}{|x-y|^s}<\infty$.
    \item $\La^1(\R^d)$ consists of all $f\in C^0(\R^d)$ such that $\|f\|_{\La^1(\R^d)}:=\sup\limits_{\R^d} |f|+\sup\limits_{x,y\in \R^d, \, x \neq y}\frac{|f(x)+f(y)-2f(\frac{x+y}2)|}{|x-y|}<\infty$.
    \item For $s>1$ recursively, $\La^s(\R^d)$ consists of all $f\in \La^{s-1}(\R^d)$ such that $\nabla f\in\La^{s-1}(\R^d)$. We define $\|f\|_{\La^s(\R^d)}:=\|f\|_{\La^{s-1}(\R^d)}+\sum_{j=1}^d \|D_j f\|_{\La^{s-1}(\R^d)}$.
    \item We define $C^\infty(\R^d):=\bigcap_{s>0} \La^s(\R^d)$ to be the space of bounded smooth functions.
\end{itemize}
\end{defn} 
\begin{defn}
Let $\Om \subset \R^d$ be a bounded Lipschitz domain. The H\"older-Zygmund space on $\Om$, denoted by $\La^s(\Om)$ for $s>0$,  is defined as 
$ \La^s (\Om) = \{ f: \exists \: \wti f \in \La^s (\R^d) \;  s.t. \; \wti f|_\Om = f \}$ equipped with the norm: 
\[
  |f|_{\La^s(U)}:= \inf_{\wti f\in \La^s (\R^d), \: \wti f|_\Om =f} | \wti f|_{\La^s (\R^d)}. 
\] 
\end{defn} 
\begin{rem}
  There is an intrinsic definition for the space $\La^s(\Om)$, namely, one which requires only that $f$ is defined in $\Om$, rather than assuming $f$ is the restriction of a function defined on the whole space. We will not use this definition in this paper. The interested reader can refer to \cite[Section 5]{Gong25}.      
\end{rem}

Similar to that of the H\"older norm, the following estimates for the H\"older-Zygmund norm:
\begin{lemma}{\cite[Lemma 3.1]{G-G24}}  \label{Lem::H-Z} 
 Let $\Om,\Om'$ be connected bounded Lipschitz domains and let $\var$ maps $\Om$ into $\Om'$. Suppose that $\| \var \|_{\Om,1} < C$. Then we have 
\begin{gather}  \label{H-Z_convexity}
 |u|_{\Om,(1-\thh) a + \thh b } \leq C_{a,b,\Om} |u|_{\Om,a}^{1-\thh} |u|_{\Om,b}^\thh, \quad 0 \leq \thh \leq 1, \quad a,b>0. 
  \\  \label{H-Z_product}
  |uv|_{\Om,a} \leq C_a (|u|_{\Om,a} \| v \|_{\Om,\ve} + \| u \|_{\Om,\ve} |v|_{D,a}), \quad a > 0; 
  \\ \nn 
  |u \circ \var|_{\Om,1} \leq C_{D, \Om'} |u|_{\Om',1} (1+ C_{1/\ve} \| \var \|_{\Om,1+\ve}^{\frac{1}{1+\ve}});  
  \\ \label{chain_rule_a>1}  
 |u \circ \var|_{\Om,a} \leq C_{a,\Om,\Om'}[C_{1/\ve} |u|_{\Om',a} \| \var \|_{\Om,1+\ve}^{\frac{1 + 2 \ve}{1+\ve}} + C_{1/\ve} \| u \|_{\Om',1+\ve} |\var|_{\Om,a} + \|u \|_{\Om',0}), \quad a > 1. 
\\ \label{chain_rule_a<1}  
|u \circ \var|_{\Om,a} \leq |u|_{\Om',a} \| \var \|_{\Om,1}^\all, \quad 0 < a < 1.  
\end{gather}
Here $C_{1/\ve}$ is a positive constant depending on $\ve$ that tends to $\infty$ as $\ve \to 0$. 
\end{lemma}
We also need the following more general chain rule estimate. The proof for H\"older norms can be found in the appendix of \cite{Gong20} and the estimate for Zygmund norms can be done similarly. We leave the details to the reader. 
\begin{lemma} \label{Lem::comp_long}  
  Let $D_i$ be a sequence of Lipschitz domains in $\R^d$, such that $\lip(D_i)$ is uniformly bounded. Let $F_i = I +f_i$ map $D_i$ into $D_{i+1}$, with $\| f_i \|_1 \leq C_0$. Then 
  \begin{gather} \label{comp_est_Holder} 
   \| u \circ F_m \circ \cdots \circ F_1 \|_{D_0,r} 
   \leq C_r^m \left( \| u \|_r + \sum_{1 \leq i \leq m} \| u \|_1 \| f_i \|_r + \| u \|_r \| f_i \|_1 \right), \quad r \geq 0; 
   \\ \label{comp_est_H-Z_norm}  
   | u \circ F_m \circ \cdots \circ F_1 |_{D_0,r}
    \leq C_r^m \left( |u|_r + \sum_{1\leq i \leq m} \|u\|_{1+\ve} |f_i|_r 
    + C_{1/\ve} |u|_r \| f_i \|_{1+\ve}^{\frac{1+2\ve}{1+\ve}} \right), \quad r >1.   
  \end{gather} 
\end{lemma} 
We now recall the definition of Besov space, which includes the H\"older-Zygmund space as a special case. 

In what follows we denote by $\Ss'(\R^d)$ the space of tempered distributions, and for an arbitrary open subset $U \subset \R^d$, we denote by $\Ss'(U):= \{ \wti f|_U: \wti f \in \Ss'(\R^d) \}$ the space of distributions in $U$ which are restrictions of tempered distributions in $\R^d$. 
\begin{defn}
  A \emph{classical Littlewood-Paley family} $\la$ is a a sequence $\la = (\la_j)_{j=0}^\infty$ of Schwartz functions defined on $\R^d$, such that the Fourier transform $\hht \la_j(\xi) = \int_{\R^d} 
\la_j (x) e^{-2 \pi i x \cdot \xi}$ satisfies 
\begin{itemize}
    \item  
    $\hht \la_0 \in C^\infty_c(B_2(\0))$ and $\hht \la_0 \equiv 1$ in $B_1(\0)$;  
    \item  
    $\hht \la_j (\xi) = \hht \la_0 (2^{-j} \xi) - \hht \la_0 (2^{-(j-1)}  \xi) $ for $j \geq 1$ and $\xi \in \R^d$.   
\end{itemize}
We denote by $\Cf = \Cf(\R^d)$ the set of all such families $\la$. 
\end{defn}

From the above definition, we see that if $\la = (\la_j)_{j=0}^\infty \in \Cf$, then $\supp \hht \la_j \subset \{ 2^{j-1} < |\xi| 
< 2^{j+1} \}$, for $j \geq 1$. 

In order to construct extension and smoothing operators on bounded Lipschitz domains, one needs the functions $\la_j$ to be supported in a cone. However, by a version of the uncertainty principle (for example, the Nazarov uncertainty principle \cite{Jam07}), $\hht \la_j$ cannot also be compactly supported. Therefore we need the following 
more general version of Littlewood-Paley family. 
\begin{defn}\label{Defn::Reg_LP}
		A \emph{regular Littlewood-Paley family} is a sequence $\phi=(\phi_j)_{j=0}^\infty$ of Schwartz functions, such that
		\begin{itemize}
			\item $\hht \phi_0 (0) =1$ and $\hht \phi _0(\xi) = 1+ O(|\xi|^\infty)$ as $\xi \to 0$. 
			\item $\hht \phi_j(\xi)= \hht \phi_0(2^{-j} \xi) - \hht \phi_0 (2^{-(j-1)} \xi)$, for $j\ge 1$.
		\end{itemize}
We denote by $\phi \in\mf{D} = \mf{D}(\R^d)$ the set of all such families $\phi$. 	
\end{defn}
Hence $\mf C \subset \mf D$. Also, if $\phi \in \mf D$, then $\sum_{j=0}^\infty \hht \phi_j = 1 $. 
\begin{defn}
    A \emph{generalized dyadic resolution} is a sequence $\psi=(\psi_j)_{j=1}^\infty$ of Schwartz functions, such that
		\begin{itemize}
		\item $ \hht \psi (\xi) = O(|\xi|^\infty)$ as $\xi \to 0$.  
		\item $\hht \psi_j(\xi)= \hht \psi_1(2^{-(j-1)}x)$, for $j\ge1$.
		\end{itemize}
We denote by $\Gf = \Gf(\R^d)$ the set of all such sequences $\psi$.  
\end{defn}
It is clear from the definition that if $\phi = (\phi_j)_{j=0}^\infty \in \Df$, then $(\phi_j)_{j=1}^\infty \in \Gf$. 

\begin{defn}
 We use $\Ss_0(\R^d)$ to denote the space of all infinite order moment vanishing Schwartz functions, that is, all $f \in \Ss(\R^d)$ such that $\int_{\R^d} x^\all f(x) \, dx = 0$ for all $\all \in \N^d$, or equivalently, $f \in \Ss(\R^d)$ such that $\hht f(\xi) 
= O(|\xi|^\infty)$ as $\xi \to 0$.   
\end{defn} 

In the case when $\la $ is a classical Littlewood-Paley family, we have the property that $\hht \la_j$ are compactly supported for $j \geq 1$, and $\supp \hht \la_j \cap \supp \hht \la_k = \emptyset$ if and only if $|j-k| \leq 2$. For a regular Littlewood-Paley family $\phi$, this is no longer true, as $\phi_j$ are merely Schwartz functions whose support are no longer compact. Nevertheless, we still have the following result which shows that for $j \neq k$, $\hht \la_j \cap \hht \la_k$ have only negligible overlaps.  

\begin{prop}{\cite[Corollary 3.7]{S-Y25}}  \label{Prop::LP_exp_decay} 
    Let $\eta_0, \thh_0 \in \Ss_0 (\R^d)$ and define $\eta_j(x):= 2^{jd} \eta_0 (2^jx)$ and $\thh_j(x):= 2^{jd}\thh_0 (2^j x)$ for $j \in \Z^+$. Then for any $M,N \geq 0$, there is $C = C(\eta,\thh, M,N) >0$ such that 
    \[
    \int_{\R^d} |\eta_j \ast \thh_k (x) | (1+ 2^{\max (j,k)}|x|)^N \, dx \leq C2^{-M|j-k|} , \quad \forall j,k \in \N.   
    \]
\end{prop} 
 Let $s \in \R$ and $1 \leq p, q \leq \infty$. For $\la \in \Cf$, the \emph{nonhomogeneous Besov norm} $\Bs^s_{p,q}(\la)$ of $u \in \Ss'(\R^d)$ is defined by
 \begin{equation} \label{Besov_def} 
   \| u \|_{\Bs^s_{p,q}(\la)} 
  =  \| (2^{js} \la_j \ast f )^{\infty}_{j=0} \|_{\ell^q( L^p)} 
  = \left( \sum_{j=0}^\infty 2^{jqs} \|\la_j \ast u \|_{L^p(\R^d)}^q \right)^\frac{1}{q}
  < \infty.  
 \end{equation}
The norm topology is independent of the choice of $\la$. In other words, for any $\la, \la'\in \Cf$, $1 \leq p,q \leq \infty$ and $s \in \R$, there is a $C = C_{\la,\la',p,q,s} >0$ such that for every $f \in \Ss'(\R^d)$,  
\[
  C^{-1} \| f \|_{\Bs_{p,q} (\la')} \leq \| f \|_{\Bs_{p,q} (\la)} \leq C \| f \|_{\Bs_{p,q} (\la')}. 
\]
The reader may refer to \cite[Prop.2.3.2]{Tr10} for the proof of this fact. We remark that one can also use a regular Littlewood-Paley family $\phi$ in the definition of the Besov norm \re{Besov_def}, and different choices of $\phi$ give rise to equivalent norms. See \cite{B-P-T96}.  

For this reason we will henceforth drop the reference to the choice of Littlewood-Paley family in the definition of the Besov norm and write it simply as $ \| \cdot \| _{\Bs^s_{p,q}}$.  
\begin{defn} 
The \emph{nonhomogeneous Besov space} $\Bs^s_{p,q}(\R^d)$ is defined by 
\[
  \Bs^s_{p,q}(\R^d) = \{ u \in \Ss'(\R^d): \| u \|_{\Bs^s_{p,q}} < \infty \}. 
\]
Let $\Om \subset \R^d$ be an arbitrary open subset. We define $\Bs^s_{p,q}(\Om):= \{  \wti f|_\Om : \wti f \in \Bs^s_{p,q}(\R^d) \}$, with norm defined by 
\[
  \| f \|_{\Bs^s_{p,q}(\Om)} = \inf_{\wti f|_{\Om} = f }  \| \wti f \|_{\Bs^s_{p,q}(\R^d)} .     
\] 
In other words, the space $\Bs^s_{p,q}(\Om)$ consists of exactly those $f \in \Ss'(\Om)$ which are the restrictions of $\Bs^s_{p,q}(\R^d)$. 
\end{defn} 

In practice, one would like to have an intrinsic definition of the space $\Bs^s_{p,q}(\Om)$. In a well-known paper \cite{Ry99}, Rychkov's showed that this is possible on a bounded bounded Lipschitz domain $\Om$. 
We now recall some useful construction from that paper. 

 \begin{note}
In $\R^d$, we use the $x_d$-directional cone $\Kb:= \{ (x',x_d): x_d > |x'| \}$ and its reflection 
\[
 -\Kb := \{ (x',x_d): x_d < - |x'| \}, \quad 
x' = (x_1, x_2, \dots, x_{d-1}). 
\]
\end{note}
\begin{defn} \label{Defn::K-LP_pair} 
  A \emph{$\Kb$-Littlewood-Paley pair} is a collection of Schwartz functions $(\phi_j, \psi_j)_{j=0}^\infty$ such that
\begin{itemize}
\item
$\phi = (\phi_j)_{j=1}^\infty$ and $\psi = (\psi_j)_{j=1}^\infty \in \Gf $. 
 \item 
$\supp \phi_j, \supp \psi_j\subset -\Kb \cap \{ x_d < - 2^{-j} \} $ for all $j \geq 0$. 
\item 
$\sum_{j=0}^\infty \phi_j = \sum_{j=0}^\infty \psi_j \ast \phi_j = \del_0$ is the Direc delta measure at $\0 \in \R^d$.  
\end{itemize}
\end{defn} 
For the construction of $\Kb$-Littlewood-Paley pair, the reader may refer to \cite[Prop 2.1]{Ry99} or \cite[Lemma 3.4]{S-Y25} for a slightly different exposition. Given such a pair $(\phi_j, \psi_j)$, 
Rychkov defines the following (universal) extension operator on a special Lipschitz domain $\om$ (see \rd{Def::Lip_dom}) : 
\[
  \Ec_\om f = \sum_{j=0}^\infty \psi_j \ast (\1_\om \cdot (\phi_j \ast f ) ), \quad f \in \Ss'(\om). 
\]
Given $f \in \Bs^s_{p,q} (\om)$, let $\wti f \in \Bs^s_{p,q} (\R^d)$ with $\wti f|_\om =f$. Since 
\[
  \phi_j \ast \wti f(x) = \int_{-\Kb} \wti f(x- y) \phi_j(y) \, dy = \int_{-\Kb} f(x- y) \phi_j(y) \, dy = \phi_j \ast f (x), \quad x \in \om,  
\]
we have 
\[
\left( \sum_{j=0}^\infty 2^{jqs} \| \phi_j \ast f \|_{L^p(\om)}^q \right)^{\frac{1}{q}} 
\leq \left( \sum_{j=0}^\infty 2^{jqs} \|\phi_j \ast \wti f \|_{L^p(\R^d)}^q \right)^{\frac{1}{q}} \approx \| \wti f \|_{\Bs^s_{p,q}(\R^d)},  
\]
where the last inequality holds by the fact that different regular Littlewood-Paley families $\phi$ give rise to equivalent norm $\| \cdot \|_{\Bs^s_{p,q}(\R^d)}$. 
Since this holds for any $\wti f$ with $\wti f|_\om = f$, by definition of the $\| \cdot \|_{\Bs^s_{p,q}(\om)}$ norm, we have  
\begin{equation} \label{equiv_norm_leq} 
   \| (2^{js} \phi_j \ast f)_{j=0}^\infty \|_{l^q(L^p(\om))} := \left( \sum_{j=0}^\infty 2^{jqs} \| \phi_j \ast f \|_{L^p(\om)}^q \right)^{\frac{1}{q}} 
\leq \| f \|_{\Bs^s_{p,q}(\om)}. 
\end{equation}
On the other hand, Rychkov proved the following important theorem for the universal extension operator $\Ec_\om$: 
\begin{prop}{\cite{Ry99}} \label{Prop::Ext_opt_spec_Lip}   
Let $\om$ be a special Lipschitz domain in $\R^d$. The operator $\Ec_\om$ satisfies  
\begin{itemize}
    \item $\Ec_\om$ defines a bounded map $\Bs^s_{p,q}(\om) \to \Bs^s_{p,q}(\R^d)$, for any $1\leq p,q \leq \infty$ and $s \in \R$.
    \item $\Ec_\om f|_{\om} = f$, for $f \in \Ss'(\om)$. 
\end{itemize}
\end{prop} 
More specifically, Rychkov showed that
\[
 \| \Ec_\om f \|_{\Bs^s_{p,q}(\R^d)} \leq C_{p,q,s} \left( \sum_{j=0}^\infty 2^{jqs} |\phi_j \ast f|_{L^p(\om)}^q \right)^{\frac{1}{q}} 
 = \| (2^{js} \phi_j \ast f)_{j=0}^\infty \|_{l^q(L^p(\om))}. 
\]
By definition, $\| f \|_{\Bs^s_{p,q}(\om)} \leq \| \Ec_\om f \|_{\Bs^s_{p,q}(\R^d)}$. Thus $\| f \|_{\Bs^s_{p,q}(\om)} \leq C_{p,q,s} \| (2^{js} \phi_j \ast f)_{j=0}^\infty \|_{l^q(L^p(\om))}$. Together with \re{equiv_norm_leq}, this implies that 
\[
  \| f \|_{\Bs^s_{p,q}(\om)} \approx_{p,q,s}  \| (2^{js} \phi_j \ast f)_{j=0}^\infty \|_{l^q(L^p(\om))}.
\]
Thus we have an intrinsic characterization for the $\| \cdot \|_{\Bs^s_{p,q}}(\om)$ norm. We will use this fact to construct smoothing operators with estimates in the H\"older-Zygmund space, which is a special Besov space.  

\begin{prop} \label{Prop::space_property}  
 Let $\Om \subset \R^d$ be either a bounded Lipschitz domain or $\R^d$. Then $\La^s(\Om) = \Bs^s_{\infty,\infty} (\Om)$, for $s >0 $. 
\end{prop}
\begin{proof}
It is well known that $\La^s(\R^d) = \Bs^s_{\infty,\infty} (\R^d)$ for $s>0$ (see for example \cite[p.~90]{Tr10}). Since $\La^s(\Om)$ and $\Bs^s_{\infty,\infty} (\Om)$ are defined as the restriction of functions in $\La^s(\R^d)$ and $\Bs^s_{\infty,\infty} (\R^d)$, we immediately get $\La^s(\Om) = \Bs^s_{\infty,\infty} (\Om)$ for $s>0$.   
\end{proof}

By using partition of unity and \rp{Prop::Ext_opt_spec_Lip}, one can define the universal extension operator on any bounded Lipschitz domain. We use the identification $\La^s(\Om) = \Bs^s_{\infty, \infty} (\Om)$. 
\begin{prop}{\cite{Ry99}} \label{Prop::Ext_opt_Lip}     
Let $\Om$ be a bounded Lipschitz domain in $\R^d$. There exists an operator $\Ec_\Om$ such that
\begin{itemize}
    \item $\Ec_\Om$ defines a bounded map $\La^s(\Om) \to \La^s(\R^d)$, for all $s >0$.
    \item $\Ec_\Om f|_{\Om} = f$, for $f \in \Ss'(\Om)$. 
\end{itemize}
\end{prop} 
For more detailed properties of the Rychkov extension operator we refer the reader to \cite{S-Y24}  

\begin{lemma}{{\cite[Lemma 2.1]{We89}, \cite[Lemma 3.3]{G-G24} }}  \label{Lem::Inv} 
  Let $F = I +f$ be a $C^1$ map from $B_r=\{x \in \R^d: \| x \| \leq r\} \subset \R^d$ into $\R^d$ with 
  \[
    f(0) = 0, \quad \| D f \|_{B_r, 0} \leq \thh < \yh. 
  \] 
  Let $r' = (1- \thh)r$. Then the range of $F$ contains $B_{r'}$ and there exists a $C^1$ inverse map $G= I+g$ which maps $B_{r'}$ injectively into $B_r$, with 
  \[
  g(0) = 0, \quad \| D g \|_{B_{r'},0} \leq 2 \| Df \|_{B_r,0}. 
  \]
  Assume further that $f \in \La^{a+1}(B_r)$. Then $g \in \La^{a+1}(B_{r'})$ and 
  \begin{gather*}
    \| Dg \|_{B_{r'}, a} \leq C_a \| Df \|_{B_r,a}, \quad a \geq 0;  
  \\ 
   |Dg|_{B_{r'},a} \leq C_a |Df|_{B_r,a} (1+C_{1/\ve} \| f \|^{\frac{1+2\ve}{1+\ve}}_{1+\ve}) \quad a > 1. 
\end{gather*} 
\end{lemma}
In practice our $f$ will have compact support in $B_r$ and we can take $r=r'$. 

The following result shows how an almost complex structure changes under transformation of the form $F = I +f$, where $I$ is the identity map. 
\begin{lemma} \label{Lem::A_choc}    
Let $\{ X_{\ov \all} \}_{\all=1}^n$ be a $C^1$ almost complex structure defined near the origin of $\R^{2n}$. 
\begin{enumerate}[(i)]
    \item 
    By an $\R$-linear change of coordinates of $\C^n$, the almost complex structure $\{ X_{\ov \all} \}_{\all=1}^n$ can be transformed into $\{ X_{\ov \all} = \pa_{\ov \all} + A_{\ov \all}^\beta \pa_\beta\}_{\all=1}^n$ with $A(0) = 0$. 
    \item 
 Let $F = I+f$ be a $C^1$ map with $f(0) = 0$ and $Df$ is small. The associated complex structure $\{ dF(X_{\ov \all} \}$ has a basis $\{ X'_{\ov \all} \}$ such that $ X'_{\ov \all} = \pa_{\ov \all} + A'^{\beta}_{\ov \all} \pa_{\beta}$, where $A'$ is given by 
 \[
  A(z)+ \pa_{\ov z} f + A(z) \pa_z f 
  = (I + \pa_{\ov z} \ov{f(z)} + A(z) \pa_z \ov{f(z)}) A' \circ F(z). 
 \]
\end{enumerate}
\end{lemma}
This is proved in \cite{We89}. For a more detailed proof the reader may also refer to \cite[Lemma 2.1]{G-G24}. 

We note that the formal integrability condition is invariant under diffeomorphism. This follows from the fact that if $[X_{\ov \all}, X_{\ov \beta}] = c^{\ov \gm}_{\ov \all \ov \beta} X_{\ov \gm}$, then $[F_\ast(X_{\ov \all}), F_\ast(X_{\ov \beta})] = (c^{\ov \gm}_{\ov \all \ov \beta} \circ F^{-1}) F_\ast(X_{\ov \gm})$.     
\begin{lemma} \label{Lem::int_cond}    
Let $ X = \{ X_{\ov \all} = \pa_{\ov \all} + A_{\ov \all}^{\beta} \pa_\beta \}_{\all =1}^n$ be a $C^1$ almost complex structure. Then $X$ is formally integrable if and only if  
\[
  \dbar A = [A, \pa A], \quad [A, \pa A] = (\pa A) A - A(\pa A)^T. 
\] 
\end{lemma}
\begin{proof} 
 Let $X_{\ov \beta} = \pa_{\ov \beta} + A^{\all}_{\ov \beta} \pa_{\all}$, $X_{\ov \gm} = \pa_{\ov \gm} + A^\eta_{\ov \gm} \pa_\eta $. The integrability condition says that $[X_{\ov \beta}, X_{\ov \gm}] \in$ $\Span X_{\ov \beta} $. By an easy computation we obtain
 \begin{align*}
 [X_{\ov \beta}, X_{\ov \gm}] &= X_{\ov \beta} X_{\ov \gm}
 - X_{\ov \gm} X_{\ov \beta}
 \\ &= \left( \pp{A^\all_{\ov \gm}}{\ov z_\beta} - \pp{A^{\all}_{\ov \beta}}{\ov z_\gm} 
 + A^\eta_{\ov \beta} \pp{A^\all_{\ov \gm}}{z_\eta} 
 - A^{\eta}_{\ov \gm} \pp{A^\all_{\ov \beta}}{z_\eta}  
 \right)\pp{}{z_\all}. 
 \end{align*}
 If $ [X_{\ov \beta}, X_{\ov \gm}] = \sum_{\nu} c_{\beta \gm}^\nu X_{\ov \nu} $, then $c^{\nu}_{\beta \gm} = 0$ for all $\nu$. Hence for each $\all, \beta, \gm$, we have
 \begin{equation} \label{int_cond_coord}   
  \pp{A^\all_{\ov \gm}}{\ov z_\beta} - \pp{A^{\all}_{\ov \beta}}{\ov z_\gm}  = 
A^\eta_{\ov \gm} \pp{A^\all_{\ov \beta}}{z_\eta} -  A^\eta_{\ov \beta} \pp{A^\all_{\ov \gm}}{z_\eta}.    
 \end{equation} 
 Now for each $\all$, we can identify $A^\all$ as a $(0,1)$-form: $A^\all = \sum_\beta A^\all_{\ov \beta} d \ov z_\beta$, then   
 \begin{gather*}
 \dbar A^\all = \sum_{\beta < \gm} (\pa_{\ov \beta} A^\all_{\ov \gm} - \pa_{\ov \gm} A^\all_{\ov \beta}) d \ov z_\beta \we d \ov z_\gm 
  \\
 \pa A^\all = \sum_\eta (\pa_\eta A^\all_{\ov \beta}) d z_\eta \we d \ov z_\beta. 
 \end{gather*}  
If we view $\dbar A^\all$ as a matrix whose $(\beta, \gm)$ -entry is $\pa_{\ov \beta} A^\all_{\ov \gm} - \pa_{\ov \gm} A^\all_{\ov \beta}$, and $\pa A^\all$ as a matrix whose $(\beta,\eta)$-entry is $\pa_\eta A^\all_{\ov \beta}$, then \re{int_cond_coord} implies that $\dbar A = (\pa A) A - A (\pa A)^T$.   
\end{proof} 
As a special case of \rl{Lem::int_cond}, if $A = \{ A_{\ov \all}^\beta \}_{\all,\beta=1}^n$ is a constant matrix, then the structure $\{X_{\ov \all}=\pa_{\ov \all} + A_{\ov \all}^{\beta} \pa_\beta\}$ is formally integrable. 
In this case, one can find an invertible linear transformation $T$ such that $T_\ast(X_{\ov \all}) \in \Span \{ \pa_{\ov \all} \} $, without assuming that the norms of $A$ to be small. 

To end the subsection, we recall the homotopy formula constructed in \cite{Gong20} and \cite{S-Y25} for a strictly pseudoconvex domain $D$ with $C^2$ boundary. 
\begin{equation} \label{Hq}   
\begin{aligned}
   \Hc_q \var(z) &= \int_{\Uc} K^0_{0, q-1}(z,\cdot) \wedge \Ec \var
+ \int_{\Uc\sm \ov D} K^{01}_{0, q-1}(z,\cdot) \wedge [\dbar, \Ec] \var, \quad [\dbar, \Ec] \var = \dbar \Ec \var - \Ec \dbar \var. 
\end{aligned} 
\end{equation}
Here $\Uc$ is a neighborhood of the closure of $D$, and $\Ec = \Ec_D$ is Rychkov's universal extension operator for the domain $D$.  

In our iteration, we shall apply the above homotopy operator to a sequence of strictly 
pseudoconvex domains $D_j$, where the neighborhood $\Uc$ is fixed, and $\dist(D_j, \Uc)$ is bounded below by some positive constant for all $j \in \N$. 
\begin{prop}{\cite{Gong20, S-Y25}} \label{Prop::dbar}    
Let $D$ be a strictly pseudoconvex domain in $\C^n$ with $C^2$ boundary, and let $\Hc_q$ be the homotopy operator given by \re{Hq}. Then the following statements are true. 
 \begin{enumerate}[(i)]
  \item $\Hc_q: \La^r_{(0,q)}(\ov D) \to  \La^{r+\yh}_{(0,q-1)}(\ov D)$, for all $r >0$.  
  \item
  $\var = \dbar \Hc_q \var + \Hc_{q+1} \dbar \var$, for any $\var \in \La^r_{(0,q)}(\ov D)$ such that $\db \var \in \La^r_{(0,q+1)}(\ov D)$. 
 \end{enumerate}
 Furthermore, the operator norm of $\Hc_q$ is stable under small $C^2$ perturbation of the domain.  
\end{prop}  

\subsection{Smoothing operator on bounded Lipschitz domains}
In this subsection we construct a Moser-type smoothing operator on bounded Lipschitz domains.  
In \cite{Mo62}, Moser constructed a smoothing operator $L_t: C^0(U_0) \to C^\infty (D_0)$, where $D_0, U_0$ are open sets in $\R^d$ such that $D_0 \subset \subset U_0$. Assume that $\dist(bD_0, b U_0) = t_0>0$. Then $L_t$ is given by 
\[
  L_t u(x) = \int_{\R^d} u(x- y) \chi_t(y) \, dy, \quad x \in D_0, \; 0 < t < \frac{t_0}{C},  
\]
where $\chi_t(z) := t^{-d} \chi (z/t)$, $\supp \chi \subset \{ x \in \R^d: |x| < 1 \}$, $ \int \chi(z) \, dz = 1 $, and 
\[
  \int z^I \chi(z) \, dz = 0 , \quad 0 < |I| \leq M, \quad M < \infty. 
\]
Moser showed that the following estimate hold for $0 < t < t_0 /C$:  
\begin{gather} \nn 
  \| L_t u \|_{D_0,r} \leq C_{r,s} t^{s-r} \| u \|_{U_0,s}, \quad 0 \leq s \leq r < \infty; 
  \\ \label{Mo_sm_op_est}
  \| (I-L_t) u \|_{D_0,s} \leq C_{r,s} t^{r-s} \| u \|_{U_0,r}, \quad r \geq 0, \; 
  s \leq r <  s+M. 
\end{gather} 
For our smoothing operator, we do not require $u$ to be defined on a large domain $U_0$. Furthermore, we note that the smoothing operator $L_t$ depends on a finite parameter $M$, whereas our smoothing operator has no such dependency and satisfies the corresponding 
estimate \re{Mo_sm_op_est} for $M = \infty$. 

\begin{prop} \label{Prop::SmOp}
Let $\Om$ be a bounded Lipschitz domain in $\R^d$. Then there exist operators $S_t: \Ss'(\Om) \to C^\infty(\Om)$ such that for all $0 < s \leq r < \infty$. 
\begin{enumerate}[(i)]
  \item  
 $ | S_t u |_{\Om, r} \leq C t^{s-r} | u |_{\Om, s}$; 
 \item 
 $| (I - S_t) u |_{\Om, s} \lesssim C' t^{r-s} | u 
|_{\Om,r}$.     
\end{enumerate}
Here the constants $C,C'$ depend only on $r,s$ and the Lipschitz norm of $\Om$. 
\end{prop}
\begin{proof} 
First we prove the statements when the domain is a special Lipschitz domain of the form $\om = \{ (x', x_d) \in \R^d: x_d > \rho(x') \}$, where $|\na \rho|_{L^\infty(\R^{d-1})}<1$. In particular, we have $\om + \Kb = \om$. $N \in \N$.
Let $(\phi_j, \psi_j)_{j=0}^\infty$ be a $\Kb$-Littlewood-Paley pair (\rd{Defn::K-LP_pair}).  We define the following smoothing operator $S_t^\Kb$ on $\Ss'(\om)$: 
\begin{equation} \label{SmOpt_L-P}   
 S_t^\Kb u:= \sum_{k =0}^{\fl{\log_2 t^{-1}}} \psi_k \ast \phi_k \ast u.   
\end{equation} 
In particular if $t \in (2^{-N-1}, 2^{-N})$, then 
the above sum becomes $\sum_{k=0}^N$. 
Using the equivalence of the H\"older-Zygmund $\La^s$ norm and the Besov $B^s_{\infty, \infty}$-norm, it suffices to prove that 
\[
 \sup_{j \in \N} 2^{jr} |\la_j \ast (S_t^\Kb u) |_{L^\infty(\om)}  
 \lesssim t^{s-r} \sup_{j \in \N} 2^{js} | \phi_j \ast u |_{L^\infty(\om)}. 
\]
where $\{ \la_j \}_{j\in \N}, \{ \phi_j \}_{j\in \N}$ are regular Littlewood-Paley families (\rd{Defn::Reg_LP}).  
We have 
\begin{align*}
 \sup_{j \in \N} 2^{jr} |\la_j \ast (S_t^\Kb u) |_{L^\infty(\om)}
&= \sup_{j \in \N} 2^{jr} \left| \la_j \ast \sum_{k=0}^N \psi_k \ast \phi_k \ast u \right|_{L^\infty(\om)}  
\\ &\leq \sup_{j \in \N} 2^{jr} \sum_{k=0}^N |\la_j \ast \psi_k \ast \phi_k \ast u|_{L^\infty(\om)} 
\\ &\leq \sup_{j \in \N} 2^{j(r-s)} 2^{js} \sum_{k=0}^N | \la_j \ast \psi_k |_{L^1(-\Kb)} | \phi_k \ast u|_{L^\infty(\om)}, 
\end{align*} 
where in the last inequality we used Young's inequality. By \rp{Prop::LP_exp_decay} with $N=0$, the last expression is bounded up to a constant multiple $C=C(M)$ by 
\[
 \sup_{j \in \N} 2^{j(r-s)} 2^{js} \sum_{k=0}^N 
2^{-M|j-k|} | \phi_k \ast u|_{L^\infty(\om)}
 = 2^{N(r-s)} \sup_{j \in \N} 2^{(j-N) (r-s)} 
 2^{js} \sum_{k=0}^N 
2^{-M|j-k|} | \phi_k \ast u|_{L^\infty(\om)}. 
\]
If $j <N$, then $2^{(j-N) (r-s)}2^{-\frac{M}{2} |j-k|} < 1$. If $j \geq N$, then $|j-k| = j-k \geq j-N$ for $k \leq N$, so $2^{-\frac{M}{2}|j-k|} \leq 2^{-\frac{M}{2}(j-N)}$. It follows that 
$2^{(j-N)(r-s)} 2^{-\frac{M}{2} |j-k|} \leq 2^{(j-N)(r-s-\frac{M}{2})} < 1$ if we choose $M > 2(r-s)$. In any case, the above estimate leads to 
\begin{equation} \label{St_est_1}  
\begin{aligned} 
 \sup_{j \in \N} 2^{jr} |\la_j \ast (S_t^\Kb u) |_{L^\infty(\om)}
&\lesssim_{r,s} 2^{N(r-s)} \sup_{j \in \N} 2^{js} \sum_{k=0}^N 2^{-\frac{M}{2} |j-k|} |\phi_k \ast u |_{L^\infty(\om)}
\\ &\leq 2^{N(r-s)} \sup_{j \in \N} \sum_{k=0}^N 2^{(s-\frac{M}{2}) |j-k|} 2^{ks} |\phi_k \ast u |_{L^\infty(\om)}, 
\end{aligned}  
\end{equation}
where we take $M> 2s$. Let 
\[
  u[a]:= 2^{|a| (s-\frac{M}{2})}, \quad a \in \Z,  \qquad v[b] :=  2^{bs} |\phi_b \ast f|_{L^\infty(\om)}, \quad b \in \N.  
\]
We denote $\left| u \right|_{l^1(\Z)}:= \sum_{a \in \Z} \left| u[a] \right|$ and $\left| v \right|_{l^\infty(\N)}:= \sup_{b \in \N} \left| u[b] \right|$. 
Then 
\begin{align*}
\sup_{j \in \N} \sum_{k=0}^N 2^{(s-\frac{M}{2}) |j-k|} 2^{ks} |\phi_k \ast u |_{L^\infty(\om)} 
 \leq |u|_{l^1(\Z)} |v|_{l^\infty(\N)} 
\lesssim_{r,s} |v|_{l^\infty(\N)} = \sup_{j \in \N} 2^{js} |\phi_j \ast f|_{L^\infty(\om)}.  
\end{align*}
Thus we get from \re{St_est_1}
\begin{align*}
  \sup_{j \in \N} 2^{jr} |\la_j \ast (S_t^\Kb u) |_{L^\infty(\om)} 
&\lesssim_{r,s} 2^{N(r-s)} \sup_{j \in \N} 2^{js} |\phi_j \ast f|_{L^\infty(\om)}
\\ &= t^{s-r} \sup_{j \in \N} 2^{js} |\phi_j \ast f|_{L^\infty(\om)}, \quad s \leq r, \quad t = 2^{-N}. 
\end{align*}
In other words, we have shown that $|S_t^\Kb u|_{\La^r(\ov \om)} \lesssim_{r,s} t^{s-r} |u|_{\La^s(\ov \om)}$, $s \leq r$. 
\\[10pt] 
(ii) 
From \re{SmOpt_L-P} we have 
\[
  (I-S_t^\Kb) u= \sum_{k > N} \psi_k \ast \phi_k \ast u. 
\] 
It suffices to show that 
\[
  \sup_{j \in \N} 2^{js} |\la_j \ast [(I-S_t^\Kb)u]|_{L^\infty(\om)} \lesssim t^{r-s} \sup_{j \in \N} 2^{js} |\phi_j \ast u|_{L^\infty(\om)}.     
\] 
We have 
\begin{align*}
  \sup_{j \in \N} 2^{js} |\la_j \ast [(I-S_t^\Kb)u]|_{L^\infty(\om)} 
&= \sup_{j \in \N} 2^{js} \left| \la_j \ast \sum_{k=N+1}^\infty \psi_k \ast \phi_k \ast u \right|_{L^\infty(\om)} 
\\ &\leq 2^{js} \sup_{j \in \N} \sum_{k=N+1}^\infty |\la_j \ast \psi_k |_{L^1(-\Kb)} |\phi_k \ast u|_{L^\infty(\om)}. 
\end{align*}
By \rp{Prop::LP_exp_decay}, the last expression is bounded up to a constant multiple $C=C(M)$ by 
\[
\sup_{j \in \N} 2^{j(s-r)} 2^{jr} \sum_{k = N+1}^\infty 
2^{-M|j-k|} | \phi_k \ast u|_{L^\infty(\om)}
 = 2^{N(s-r)} \sup_{j \in \N} 2^{(j-N) (s-r)} 
 2^{js} \sum_{k=N+1}^\infty 
2^{-M|j-k|} | \phi_k \ast u|_{L^\infty(\om)}. 
\] 
If $j \geq N$, then $2^{(j-N)(s-r)} \leq 1$. If $j < N$, then $ -(j-N)= |j-N| \leq |j-k|$ for all $k \geq N+1$. Hence $2^{(j-N)(s-r)-\frac{M}{2}|j-k|} \leq 2^{-(j-N)(r-s) - \frac{M}{2}|j-k|} \leq 2^{|j-k|(r-s-\frac{M}{2})} < 1$, where we choose $M > 2(r-s)$. In all cases, we get from the above estimates that 
\begin{align*}
 \sup_{j \in \N} 2^{js} |\la_j \ast [(I-S_t^\Kb)u]|_{L^\infty(\om)} &\lesssim_{r,s} 2^{N(s-r)} \sup_{j \in \N} 2^{jr} \sum_{k=N+1}^\infty 2^{-\frac{M}{2} |j-k|} |\phi_k \ast u|_{L^\infty(\om)}, \quad t = 2^{-N}. 
\end{align*}
The rest of the estimates follow identically as in (i), and consequently we prove that $|(I-S_t^\Kb)u|_{\om,s}  \lesssim_{r,s} t^{r-s} |u|_{\om, r}$ for $0 < s \leq r $.   

Finally we prove both (i) and (ii) for general bounded Lipschitz domains. For this we use partition of unity. Take an open covering $\{ U_{\nu} \}_{\nu=0}^M$ of $\Om$ such that 
\[
   U_0  \subset\subset \Om, \quad b \Om \seq \bigcup_{ \nu=1}^M U_{\nu},  \quad U_{\nu} \cap \Om = U_{\nu} \cap \Phi_\nu( \om_{\nu}) ,  \quad \nu=1, \dots, M. 
\]
Here each $\om_{\nu}$ is a special Lipschitz domain of the form $\om_{\nu} = \{ x_d > \rho_\nu (x') \} $, with $|\rho_\nu|_{L^\infty(\R^{d-1})} < 1$, and $\Phi_\nu$, $1\le \nu\le M$ are invertible affine linear transformations. Here we note that $|D\Phi_\nu|_{L^\infty}$ is bounded (up to a constant) by the Lipschitz norm of $\Om$.  

If $f$ has compact support in $\Om$, we define the smoothing operator $S^0_t$ by  
\[
  S^0_t f = \sum_{k=0}^N \eta_k \ast \thh_k \ast f, \quad t= 2^{-N}.  
\] 
Here we can choose any Littlewood-Paley pair $(\thh_j, \eta_j)_{j=0}^\infty$ with $\thh \in \Df$, $\eta \in \Gf$ and $\sum_{j=0}^\infty \thh_j = \sum_{j=0}^\infty \eta_j \ast \thh_j = \del_0$. Then the same proof as above shows that $| S_t^0 f|_{\Om, r} \lesssim t^{s-r} | f |_{\Om, s}$ and $| (I-S_t^0) f |_{\Om,s} \lesssim t^{r-s} | f |_{\Om,r}$ for $0 < s \leq r$.

Fix a partition of unity $\{ \chi_{\nu} \}_{\nu=0}^M$ associated with $\{U_{\nu} \}_{\nu=0}^M$, such that $\chi_{\nu} \in C^{\infty}_c(U_{\nu}) $ and $\chi_0 + \sum_{\nu=1}^M \chi_{\nu}^2 = 1$. For each $1 \leq \nu \leq M$, we have the property $\om_{\nu} + \Kb = \om_{\nu}$, where $\Kb:= \{ x\in \R^d: x_d >|x'| \}$. Let $S_t^\Kb$ be given as above, we define 
\begin{equation} \label{St_formula}
  S_t u := S^0_t (\chi_0 u) + \sum_{\nu=1}^M \chi_\nu S_t^\nu ( \chi_\nu u),   
\end{equation}
where $ S^\nu_t g:= [ S_t^\Kb(g \circ \Phi_\nu)] \circ \Phi_\nu^{-1} $, $1 \leq \nu \leq M$. 

Applying the estimates for $S_t^0$ and $S_t^\Kb$, we get
\begin{align*}
  | S_t u |_{\Om, r}  
  &\lesssim t^{s-r} \left( | \chi_0 u |_{\Om, s} + | 
(\chi_\nu u) \circ \Phi_\nu |_{\om, s} \right)  
 \lesssim t^{s-r} | u |_{\Om,s}, \quad 0 < s \leq r, 
  \end{align*}
where the constant depends only on $r,s$ and $\lip(\Om)$.

On the other hand, since $u = \chi_0 u + \sum_{\nu=1}^M \chi_\nu^2 u$ we have 
\begin{align*}
  | (I - S_t) u |_{\Om, s}
  &\leq | (I - S^0_t) (\chi_0 u) |_{\Om, r} 
  + \sum_{\nu=1}^M |  \chi_\nu (I - S^\nu_t)(\chi_\nu u) |_{\Om, s} 
  \\ &\lesssim t^{r-s} \left( | \chi_0 u |_{\Om, r} 
  + | (\chi_\nu u) \circ \Phi_\nu |_{\om,r} \right)
\lesssim t^{r-s} | u |_{\Om, r}. 
  \quad 0 < s \leq r,  
\end{align*}
where the constant depends only on $r,s$ and $\lip(\Om)$. 
\end{proof}  
\begin{lemma} \label{Lem:: D_St_nu_comm}   
 Let $\Om$ be a bounded Lipschitz domain in $\R^d$ and $\{S_t^\nu\}_{\nu=1}^M$ be the smoothing operator constructed in the proof of \rp{Prop::SmOp}. Denote $\pa_i = 
\pp{}{x_i}$, $i = 1,2,\dots, d$. Then $[\pa_i,S_t^\nu] (\chi_\nu f) \equiv 0$ for all $f \in \La^r(\ov \Om)$, $r > 1$.  
\end{lemma}
\begin{proof}
 Recall that $S_t^\nu g:= [S^\Kb_t (g \circ \Phi_\nu)] \circ \Phi_\nu^{-1} $, where $S^{\Kb}_t$ is the smoothing operator defined on $\Ss'(\om_\nu)$ and given by \re{SmOpt_L-P}. Hence we have  
 \begin{align*}
  [\pa_i,S_t^\nu] (\chi_\nu f) 
  &= \pa_i(S_t^\nu (\chi_\nu f)) - S_t^\nu (\pa_i(\chi_\nu f)) 
  \\ &= \pa_i( S^\Kb_t [(\chi_\nu f) \circ \Phi_\nu] \circ \Phi_\nu^{-1}) - S^\Kb_t [\pa_i(\chi_v f) \circ \Phi_\nu] \circ \Phi_\nu^{-1} 
  \\ &= (\na S^\Kb_t[(\chi_\nu f) \circ \Phi_\nu] ) \circ \Phi_\nu^{-1} \cdot \pa_i \Phi_\nu^{-1}  
- S^\Kb_t[\na ( (\chi_\nu f) \circ \Phi_\nu)] \circ \Phi_\nu^{-1} \cdot \pa_i\Phi_\nu^{-1} 
\\ &\quad + S^\Kb_t[\na ( (\chi_\nu f) \circ \Phi_\nu)] \circ \Phi_\nu^{-1} \cdot \pa_i \Phi_\nu^{-1} - S^\Kb_t [\pa_i (\chi_\nu f) \circ \Phi_\nu ] \circ \Phi_\nu^{-1}
\\ &= [\na, S^\Kb_t] ( (\chi_\nu f) \circ \Phi_\nu) \circ \Phi_\nu^{-1} \cdot \pa_i\Phi_\nu^{-1}, 
 \end{align*}
where in the last step we used the fact that $\Phi_\nu$ is a linear transformation so that 
\begin{align*} 
S^\Kb_t[\na ( (\chi_{\nu} f) \circ \Phi_{\nu})] \circ \Phi_{\nu}^{-1} \cdot \pa_i \Phi_{\nu}^{-1} 
&= \begin{pmatrix}
 \sum_{j=1}^n \pp{\Phi_\nu^j}{x_1} S^\Kb_t \left(\pp{(\chi_\nu f)}{y_j} \circ \Phi_\nu \right)  
 \\  
 \vdots
 \\ 
\sum_{j=1}^n \pp{\Phi_\nu^j}{x_n} S^\Kb_t \left(\pp{(\chi_\nu f)}{y_j} \circ \Phi_\nu \right) 
\end{pmatrix} 
\circ \Phi_\nu^{-1} 
\cdot 
\begin{pmatrix}
  \pa_i (\Phi_\nu^{-1})^1
  \\
  \vdots 
  \\ 
  \pa_i (\Phi_\nu^{-1})^n
\end{pmatrix}
\\&= \sum_{j=1}^n S^\Kb_t \left(\pp{(\chi_\nu f)}{y_j} \circ \Phi_\nu \right) \circ \Phi_\nu^{-1} 
\cdot \left( \sum_{\all=1}^n 
\pp{\Phi^j_\nu}{x_\all} \circ \Phi_\nu^{-1} \cdot \pp{(\Phi^{-1}_\nu)^\all}{y_i} \right) 
\\ &= S^\Kb_t \left(\pp{(\chi_\nu f)}{y_i} \circ \Phi_\nu \right) \circ \Phi_\nu^{-1}. 
\end{align*} 
Now since $S^\Kb_t u = \sum_{k=0}^{\fl{\log_2 t^{-1}}} \psi_k \ast \phi_k \ast u$ is a convolution operator, we have $[\na, S^\Kb_t] \equiv 0$ on $\om_\nu$. Thus $[\pa_i,S_t^\nu](\chi_\nu f) \equiv 0$. 
\end{proof}

\begin{prop} \label{Prop::D_St_comm_est} 
 Let $\Om$ be a bounded Lipschitz domain in $\R^d$ and let $S_t$ be the smoothing operator constructed in the proof of \rp{Prop::SmOp}. Denote $\pa_i = 
\pp{}{x_i}$, $i = 1,2,\dots, d$. Then for all $u \in \La^r( \ov \Om)$ with $r > 1$, the following holds  
 \begin{equation} \label{D_St_comm_est} 
    \left| [\pa_i, S_t] u \right|_{\Om,s} \leq C t^{r-s} | u |_{\Om,r}, \quad 0 < s \leq r.  
 \end{equation}
Here the constant $C$ depends only $r,s$ and the Lipschitz norm of the domain.  
\end{prop}  
\begin{proof}
  By the formula for $S_t$ \re{St_formula}, we can write 
  \begin{align*}
    [\pa_i, S_t] u 
    &= \pa_i S_t u - S_t \pa_i u
    \\ 
    &= \pa_i S_t^0(\chi_0 u) + \sum _{\nu=1}^M 
    \pa_i [\chi_\nu S^\nu_t (\chi_\nu u)] 
    - S_t^0 (\chi_0(\pa_i u )) - \sum_{\nu=1}^M \chi_\nu S_t^\nu (\chi_\nu (\pa_i u))
    \\ &= \{ \pa_i S_t^0(\chi_0 u) - S_t^0 \pa_i(\chi_0 u) \} + S_t^0( (\pa_i \chi_0) u) 
    + \sum_{\nu=1}^M (\pa_i \chi_\nu) S_t^\nu (\chi_\nu u) 
    \\ &\quad+ \sum_{\nu=1}^M \chi_\nu \pa_i S^\nu_t(\chi_\nu u) - \chi_\nu S_t^\nu \pa_i(\chi_\nu u) + \sum_{\nu=1}^M \chi_\nu S_t^\nu((\pa_i \chi_\nu)u)
    \\ &= S_t^0 ((\pa_i \chi_0) u) + \sum_{\nu=1}^M (\pa_i \chi_\nu) S_t^\nu (\chi_\nu u) + \sum_{\nu=1}^M \chi_\nu S_t^\nu((\pa_i \chi_\nu)u). 
  \end{align*}
Here to get the last line we use $[\pa_i, S_t^0](\chi_0 u) \equiv 0$ and \rl{Lem:: D_St_nu_comm}. 
Since $0 = \pa_i(\chi_0 + \sum_{\nu=1}^M \chi_\nu^2) = \pa_i \chi_0 + 2 \sum_{\nu=1}^M \chi_\nu \pa_i(\chi_{\nu})$, we can write 
\begin{align*}
  [\pa_i, S_t] u &= (S_t^0 - I)((\pa_i \chi_0) u) + \sum_{\nu=1}^M (\pa_i \chi_\nu) (S_t^\nu - I) (\chi_\nu u) + \sum_{\nu=1}^M \chi_\nu (S_t^\nu - I) ((\pa_i \chi_\nu) u). 
\end{align*}
Applying the proof of \rp{Prop::SmOp} to the right-hand side above we get \re{D_St_comm_est}.  
\end{proof}

\subsection{Stability of constants} 
For our application, we need to construct a sequence of domains $\{ D_{j+1} = F_{j} (D_{j}) \}_{j=0}^\infty$, where $\{F_j \}_{j=0}^\infty$ is a sequence of diffeomorphisms constructed using the above defined smoothing and homotopy operators.  
For the iteration to work, we need to make sure that each $D_j$ is strictly pseudoconvex with $C^2$ boundary, and also that the maps $F_j \circ F_{j-1} \circ \cdots \circ F_1$ converge to a limiting map $F_\infty$ in the desired norms. This requires the stability of constants in all the estimates under small $C^2$ perturbation of domains. We now make precise this notion of stability, following \cite{G-G24}.  

Let $D_0 = \{ x \in \Uc: \rho_0 <0 \} \subset \Uc \subset \R^d$ be a domain with $C^2$ boundary, where $\Uc$ is a neighborhood of $\ov D_0$ and $\rho_0$ is a $C^2$ defining function of $D_0$. Let
\[
  G_{\ve_0} = \{ \rho \in C^2(\Uc): \| \rho - \rho_0 \|_{\Uc,2} \leq \ve_0 \}. 
\]
Here $\ve_0$ is a small positive number such that for all $\rho \in G_{\ve_0}$, we have $d \rho(x) \neq 0$ on $\{ x\in \Uc: \rho (x) = 0 \}$.  

\begin{defn} \label{Def::upp_low_stable}
We say that a function $\Ac: G_{\ve_0} \to (0,\infty)$ is \emph{upper stable} (resp. \emph{lower stable}) under (small) $C^2$ perturbation of the domain, if there exists $\ve(D_0) >0$ and a constant $C_0 (D_0) > 1$, such that 
\[
  \Ac(\rho) \leq C_0 (D_0) \Ac(\rho_0) \quad (\text{resp.} \:  \Ac(\rho_0) \leq C_0 (D_0) \Ac(\rho) ). 
\]
for all $\rho$ satisfying $\| \rho - \rho_0 \|_{\Uc,2} \leq \ve(D_0)$. 
\end{defn}

We make note of the following examples of upper stable mappings which are relevant to our proof. 

\begin{enumerate}
  \item 
    The constants appearing in \rl{Lem::H-Z},   \rl{Lem::comp_long} are upper stable under small $C^2$ perturbation of the domain.  
    \item 
    The operator norms of Rychkov's extension operator (\rp{Prop::Ext_opt_Lip}) depend only on the Lipschitz norm of the domain, which is upper stable under small $C^1$ perturbation of the domain.  
    \item 
  The operator norms for the smoothing operator are upper stable under small $C^1$ perturbation of the domain $D_0$ (see \rp{Prop::SmOp}).  
  \item The operator norms for the homotopy operator \re{Hq} are upper stable under small $C^2$ perturbation of the domain. 
\end{enumerate}

\medskip 

We now show how \rt{Thm::intro_main_thm} implies  \rt{Thm::Loc_N-N}. The proof is almost identical to the one for \cite[Theorem 1.2]{G-G24}, and we include it here for the reader's convenience. The lower stability of $\del_0$ plays a key role in the proof. 
Let $M \subset bU$ be a $C^2$ strictly pseudoconvex real hypersurface, and suppose that $0 \in M$ and $A(z) = o(|z|)$. By a local polynomial change of coordinates (see \cite[Lemma 2.3.]{G-G24}) that preserves the condition $A(z) = o(|z|)$, there exists a defining function $\rho$ for $M$, defined near the origin, such that $\rho<0$ on $U$, $\rho = 0$ on $M$, and 
\begin{equation} \label{rho_poly_chofco}   
  \rho(z) = - y_n + |z'|^2 + h(z', x_n), \quad z_n = x_n + i y_n,    
\end{equation} 
where $h = o(2)$ is a $C^2$ function. 

We shall need the following result of Gan and Gong.  
\begin{prop}{\cite[Proposition 2.4]{G-G24}} \label{Prop::dilation}   
  Let $M \subset b U$ be a $C^2$ strictly pseudoconvex real hypersurface containing the origin, which has a local defining function of the form \re{rho_poly_chofco}. Let $\{ X_{\ov \all} =  
\pa_{\ov \all} + A^\beta_{\ov \all} \pa_\beta \}_{\all=1}^n$ define an integrable almost complex structure on the one-sided domain $U \cup M$ with $A(z) = o(|z|)$. Suppose that $A \in 
\La^r(U \cup M)$, $1<r<\infty$. Then after a non-isotropic dilation $\phi_\la(z', z_n) = (\la^{-1} z', \la^{-2} z_n)$, where $\ve >0$ is sufficiently small, the following hold: 
  \begin{enumerate}[(i)] 
      \item 
      There exist some open set $B \subset \C^n$ and a $C^2$ function $\rho_\la: B \to \R$ such that $D_\la = \{ z \in B: \rho_\la < 0 \} \subset \phi_\la(U \cup M)$ is a connected $C^2$ strictly pseudoconvex domain that shares part of the boundary with $\phi_\la (M)$ near the origin. Moreover, there exists a $C^2$ function $\rho_0: B \to \R$ such that $\lim_{\la \to 0} \| \rho_\la - \rho_0 \|_{B,2} = 0$ and $D_0:= \{ z \in B: \rho_0 < 0\}$ is also a connected $C^2$ strictly pseudoconvex domain.    
      \item 
      On each $\ov D_\la$, $d\phi_\la (X_{\ov \all})$ is spanned by $ \{ \pa_{\ov \all} + \left( A^{(\la)}\right)_{\ov \all}^\beta \pa_\beta \}_{\all=1}^n$, where $|A^{(\la)}|_{D_\la,r}$ tends to $0$ with $\la$.
    \end{enumerate}
\end{prop}
\medskip 
\nid 
\textit{Proof of \rt{Thm::Loc_N-N}.} Let $X_{\ov \all} \in \La^m(U \cup M)$, for $m > 3/2$. Apply \rp{Prop::dilation} to $\{ U \cup M, \{ X_{\ov \all} \}_{\all=1}^n \}$, with $\ve>0$ is  to be determined. Then we obtain a $C^2$ strictly pseudoconvex domain $D_\la \subset \phi_\la (U \cup M)$, which shares part of the boundary with $\phi_\la(M)$, and $0 \in b D_\la$. The vector fields $\{ X^{(\la)}_\all = \pa_{\ov \all} + \left( A^{(\la)}\right)_{\ov \all}^\beta \pa_\beta \}_{\all=1}^n$ define a formally integrable almost complex structure on $D_\la$ and $|A^{(\la)}|_{D_\la,m} $ tends to $0$ as $\la \to 0$. 

By \rt{Thm::intro_main_thm}, there exists $\del_0= \del_0(D_0, |A|_{\frac{3}{2} + \wti \ve_0}, \wti \ve_0 ) $ which is lower stable under a small $C^2$ perturbation of $D_0$ (Note that $\del_0$ blows up as $m \to \frac32^+$.)  Therefore, we can find $\ve>0$ sufficiently small such that 
\[
  |A^{(\la)} |_{D_\la, \frac32+ \wti \ve_0} \leq \del_0/ C_0(D_0) \leq \del_\la, 
  \quad 
   m > \frac{3}{2} + \wti \ve_0.   
\]
where $\del_\la$ denotes the constant in the hypothesis of \rt{Thm::intro_main_thm} for the domain $D_\la$. Consequently, by applying \rt{Thm::intro_main_thm}, we obtain a diffeomorphism $F_\la: D_\la \to \C^n$ that sends the almost complex structure to the standard one, such that $F_\la \in \La^{m+\yh^-}(D_\la)$ if $m<\infty$, and $F_\la \in \La^\infty(D_\la)$ if $m= \infty$. Since $D_\la$ shares part of the boundary with $\phi_\la(M)$, $F_\la$ induces a diffeomorphism near $0 \in M$ that sends the almost complex structure to the standard one on one side of the domain. We can then take the embedding to be $F = F_\la \circ \phi_\la$. 
\bigskip 
\section{Transformation of the structure under diffeomorphism}

Let $D_0$ be a strictly pseudoconvex domain in $\C^n$. Given the initial integrable almost complex structure on $D_0$, which is given by the vector fields 
$ \{ X_{\ov \all} \}_{\all=1}^n = \{ \pa_{\ov \all} + A_{\ov \all}^\beta \pa_{\beta} \}_{\all=1}^n$, we want to find a transformation $F$ defined on $\ov D_0$ that transforms the structure to a new structure closer to the standard complex structure while $\ov D_0$ is transformed to a new domain that is still strictly pseudoconvex. We shall assume the following initial condition for $A = [A_{\ov \all}^\beta]_{1\leq \all, \beta \leq n}$: 
\begin{equation} \label{A_ic}  
  t^{-\yh} |A|_{D_0,s} < 1, \quad s = 1 + \e_0, 
\end{equation}
where $t$ is the parameter of the smoothing operator which we will choose to be sufficiently small, and $\e_0$ can be taken to be any sufficiently small positive number to be specified later.   
We take the map in the form $F = I + f$. Applying the
extension operator $\Ec_{D_0}$ to $f$ (\rp{Prop::Ext_opt_Lip}), we can assume that $f$ is defined with compact support on some open set $\Uc_0$ containing $\ov D_0$. Let $B_R = \{ z \in \C^n: |z| < R \}$, where $R$ is very large such that 
\[
 D_0 \subset \subset \Uc_0 \subset \subset B_{R/2}. 
\] 

We will use $C_m$ (resp. $C_s,C_r$ etc.) to denote a constant depending on $m$, and which is upper stable under small $C^2$ perturbation of the domain $D_0$. We will use the same $C_m$ to denote different constants depending on $m$. 

As in the proof of \rl{Lem::int_cond}, we regard $A^\all_\beta$ as the coefficients of the $(0,1)$ form 
$A^\all:= A^\all_\beta d \ov z_\beta$. 
We then apply the homotopy formula component-wise to $A = 
(A^1, \dots, A^n)$ on $\ov D_0$ so that 
\[
 A = \dbar P A + Q \dbar A,  
\]
where $P=\Hc_1$ and $Q = \Hc_2$ are given by formula \re{Hq}. By \rp{Prop::dbar}, we have 
\begin{equation} \label{PQ_est} 
   |PA|_{D_0,r + \yh}, \; |Q A|_{D_0,r + \yh} < |A|_{D_0,r}, \quad r >0.     
\end{equation}

We set $f= - \Ec_{D_0} S_t PA$, where $S_t$ is the smoothing operator constructed in \rp{Prop::SmOp}. 
By \rp{Prop::SmOp} (i) and \re{PQ_est}, we have the following estimates for $f$: 
\begin{gather} \label{f_m_norm_est} 
  |f|_{B_R,m} \leq C_m |S_t PA|_{D_0,m} \leq C_m' |PA|_{D_0,m} \leq C_m''|A|_{D_0, m-\yh}, \quad m > \yh;
  \\ \label{f_m_norm_est_2}  
 |f|_{B_R,m} \leq C_m |S_t PA|_{D_0,m}
\leq C_m' t^{-\yh} |PA|_{D_0,m-\yh} \leq C_m'' |A|_{D_0, m-1}, \quad m > 1. 
\end{gather}
In view of \re{f_m_norm_est} and the initial condition \re{A_ic}, we have 
\[
   \| f \|_{B_R,1} 
  \leq |f|_{B_R,1+\ve} \leq C_1  
|A|_{D_0,\yh+\ve} \leq C_1 t^{\yh} < \yh, 
\]
where we choose $t < \frac{1}{(2C_1)^2}$.  
By \rl{Lem::Inv}, $F$ is a diffeomorphism from $B_R$ onto itself, where $R$ is a sufficiently large number and $\Uc \subset \subset B_{R/2}$.  Furthermore, 
\begin{gather*} 
\| g \|_{B_R, a} \leq C_a \| f \|_{B_R,a}, \quad a \geq 1; 
\\  
| g |_{B_R, a} \leq C_a | f |_{B_R,a}, \quad a > 2, 
\end{gather*} 
which together with \re{f_m_norm_est} implies 
\begin{equation} \label{g_m_norm_est} 
  |g|_{B_R,m} \leq C_m |f|_{B_R,m} \leq C'_m |A|_{D_0,m-\yh}, \quad m > 1. 
\end{equation}

By \rl{Lem::A_choc}, 
the new structure takes the form 
\begin{equation} \label{A'_circ_F}
  A' \circ F = (I + \dbar \ov f + A \dbar f)^{-1} (A + \dbar f + A \pa f).  
\end{equation} 
Substituting $f = - S_t PA$, we have 
\begin{align*} 
  A + \dbar f + A \pa f
  &= A - \dbar(S_t PA) + A \pa f \\
  &= A - S_t \dbar PA + [S_t, \dbar] PA   + A \pa f \\ 
  &= A - S_t(A - Q \dbar A) + [S_t, \dbar] PA  + A \pa f \\
  &= (I-S_t)A + S_t Q \dbar A + [S_t, \dbar] PA + A \pa f. 
\end{align*}
We shall use the following notation: 
\[ 
\Kc= \ov{\pa f} + A \dbar f, \quad I_1 = (I - S_t)A, 
\quad I_2 = S_t Q \dbar A, \quad I_3 = [S_t, \dbar] PA, \quad I_4 = A \pa f, 
\] 
and consequently we can rewrite \re{A'_circ_F} as 
\[ 
  \wti A = (I + \Kc)^{-1} (I_1 + I_2 + I_3 + I_4), \quad \wti A = A' \circ F. 
\]
We first estimate the $I_j$-s. By \rp{Prop::SmOp}, we get 
\begin{equation} \label{I1_est} 
 |I_1|_{D_0,m} = |(I - S_t)A|_{D_0,m}  
  \leq C_{m,r} t^{r-m} |A|_{D_0,r}, \quad 0 < m \leq r.  
\end{equation} 
Substituting $s$ for $m$ in \re{I1_m_norm_est} we have  
\begin{equation} \label{I1_s_norm_est} 
  |I_1|_{D_0,s} \leq C_{r,s} t^{r-s} | A|_{D_0,r}, \quad 0 < s \leq r. 
\end{equation}
Substituting $m$ for $r$ in \re{I1_m_norm_est} we have 
\begin{equation} \label{I1_m_norm_est}  
  |I_1|_{D_0,m} 
  \leq C_m |A|_{D_0, m}, \quad m > 0 
\end{equation} 
For $I_2$, 
we need to use the integrability condition $\dbar A = [\pa A, A]$. Together with \rp{Prop::SmOp}, \re{H-Z_product}, and \re{PQ_est} to get  
\begin{equation} \label{I2_est}  
  \begin{aligned} 
   |I_2|_{D_0,m} &= |S_t Q \dbar A|_{D_0,m}
  \leq C_m t^{-\yh} |Q\dbar A|_{D_0,m-\yh}
  \\ &\leq C_m t^{-\yh} |\dbar A|_{D_0, m-1}
  = C_m' t^{-\yh} \left|[\pa A, A]\right|_{D_0, m-1}
  \\ &\leq C_m'' t^{-\yh} \left(|A|_{D_0,m-1} |\pa A|_{D_0,\ve} + |\pa A|_{D_0, m-1} |A|_{D_0,\ve} \right) 
  \\ &\leq 2C_m'' t^{-\yh} |A |_{D_0,s} |A |_{D_0,m}, \quad s, m > 1,  
\end{aligned} 
\end{equation} 
where we choose any $\ve \in (0, \e_0)$.  
Substituting $s$ for $m$ in the above estimate we get 
\begin{equation} \label{I2_s_norm_est} 
  | I_2|_{D_0,s} \leq C_s t^{-\yh} | A|^2_{D_0,s}, \quad s=1+\e_0.  
\end{equation}   
Using the initial condition \re{A_ic} in \re{I2_est} we get 
\begin{equation} \label{I2_m_norm_est}
| I_2|_{D_0,m} \leq C_m |A|_{D_0,m}, \quad m > 1. 
\end{equation}   

For $I_3$, we apply \rp{Prop::D_St_comm_est} and \re{PQ_est} to get 
\begin{equation} \label{I3_est} 
  \begin{gathered} 
  |I_3|_{D_0,m} = |[S_t, \dbar] PA|_{D_0,m} 
  \leq C_{m,r} t^{r+\yh-m} |PA|_{D_0,r+\yh} \leq C'_{m,r} t^{r+\yh-m} |A|_{D_0, r}, 
  \\ r \geq \yh, \quad  0 < m \leq r+\yh . 
\end{gathered}    
\end{equation}
Substituting $s$ for $m$ in the above estimate we have
\begin{equation} \label{I3_s_norm_est} 
  |I_3|_{D_0,s} \leq C_{r,s} t^{r+\yh-s} |A|_{D_0,r},  \quad 
  0 < s \leq r+\yh, \quad r \geq \yh.  
\end{equation} 
Substituting $m$ for $r$ in \re{I3_est} we have 
\begin{equation} \label{I3_m_norm_est}
   |I_3|_{D_0,m} \leq C_m |A|_{D_0,m}, \quad m \geq \yh. 
\end{equation} 

To estimate $I_4$, we recall that $f = - S_t P A$. Hence for $s,m >0$, we have
\begin{equation} \label{I4_est}
\begin{aligned} 
 |I_4|_{D_0,m} = |A \pa f|_{D_0,m} 
 &\leq C_m \left( |A|_{D_0,m} |\pa f|_{D_0,\ve} + |A|_{D_0,\ve} | \pa f |_{D_0, m} \right) 
 \\ &\leq C_m \left( |A|_{D_0,m} |f|_{D_0,1+\ve} + |A|_{D_0,\ve} | f |_{D_0, m+1} \right)
 \\ &\leq C_m' \left( |A|_{D_0,m} t^{-\yh} |A|_{D_0,\ve} + |A|_{D_0, \ve} t^{-\yh} |A|_{D_0,m} \right)
 \\ &\leq 2 C_m' t^{-\yh} |A|_{D_0,s} |A|_{D_0, m},
\end{aligned} 
\end{equation} 
where we used that $|f|_{D_0, 1+\ve} = |S_t PA|_{D_0, 1+\ve} 
\leq C_1 t^{-\yh} |PA|_{D_0,\yh+\ve} \leq C_1' t^{-\yh} |A|_{D_0,\ve}$, and similarly $|f|_{D_0,m+1} \leq C_m t^{-\yh} |A|_{D_0,m}$ for any $m > 0$.  
Applying estimate \re{I4_est} with $s$ in place of $m$ we get 
\begin{equation} \label{I4_s_norm_est}  
  |I_4|_{D_0,s} \leq C_s t^{-\yh} |A|^2_{D_0,s}, \quad s = 1 + \e_0. 
\end{equation} 
Alternatively, by using the initial condition \re{A_ic} in \re{I4_est}, we have 
\begin{equation} \label{I4_m_norm_est} 
  |I_4|_{D_0,m} \leq C_m |A|_{D_0,m}, \quad m > 0. 
\end{equation} 
Next, we estimate the low and high-order norms of $ (I + \Kc)^{-1}$, where $\Kc:= \ov {\pa f} + A \dbar f$. By using $f = -S_t PA$, the product rule \re{H-Z_product}, and \rp{Prop::SmOp} (i) we have
\begin{equation} \label{Kc_est}   
\begin{aligned} 
  |\Kc|_{D_0,m} &\leq | \dbar \ov{S_t PA}|_{D_0,m} + |A \dbar S_t PA|_{D_0,m} 
  \\ &\leq |S_t PA|_{D_0,m+1} + |A|_{D_0,m} |\dbar S_t PA|_{D_0,\e_0} 
  + |A|_{D_0,\e_0} |\dbar S_t PA|_{D_0,m}  
  \\ &\leq C_m t^{-\yh} |A|_{D_0,m}, \quad m >0 
 \end{aligned} 
\end{equation}
 where in the last inequality we used 
\begin{gather*} 
  |\dbar S_t PA|_{D_0,\e_0} \leq |S_t P A|_{D_0,1+\e_0}  
\leq C_1 |A|_{D_0,1+\ve_0} \leq C_1 t^{\yh}.  
 \end{gather*} 
Applying estimate \re{Kc_est} with $m=s$ and using the initial condition \re{A_ic} we get
\begin{equation} \label{K_s_norm_est}
  |\Kc|_{D_0,s} \leq C_s, \quad s = 1 + \e_0.  
\end{equation} 
We now consider $(I+\Kc)^{-1}$. Using the formula $(I+ \Kc)^{-1} = [\det (I+\Kc)]^{-1} B$, where $B$ is the adjugate matrix of $I + \Kc$, we see that every entry in $(I + \Kc)^{-1}$ is a polynomial in $[\det (I+\Kc)]^{-1}$ and entries of $\Kc$. By using \re{H-Z_product} and \re{Kc_est}, we get  
\begin{equation} \label{I+K_inv_m_norm}
  |(I+\Kc)^{-1} |_{D_0,m} \leq C_m (1+ |\Kc|_{D_0,m})
\leq C_m' (1+ t^{-\yh} |A|_{D_0,m}), \quad m > 0. 
\end{equation} 
In particular by the initial condition \re{A_ic}, we have 
\begin{equation} \label{I+K_inv_s_norm} 
  |(I+ \Kc)^{-1} |_{D_0,s} \leq C_s, \quad s = 1 + \e_0.  
\end{equation}
We now estimate the $s$-norm of $\wti A = (I+\Kc)^{-1}(\sum_{j=1}^4 I_j)$. Applying the product estimate \re{H-Z_product} and \re{I1_s_norm_est}, \re{I1_m_norm_est}, \re{I+K_inv_m_norm}, \re{I+K_inv_s_norm}, we get
\begin{equation} \label{I+K_inv_I1_m_norm}  
\begin{aligned}
 |(I+\Kc)^{-1} I_1|_{D_0,m}  
 &\leq C_m \left( |(I+\Kc)^{-1}|_{D_0,m} |I_1|_{D_0,\ve}   
 + |(I+\Kc)^{-1}|_{D_0,\ve} |I_1|_{D_0,m} \right) 
 \\& \leq C_m' (1+t^{-\yh}|A|_{D_0,m}) (C_s t^{s-\ve} |A|_{D_0,s}) + C_m' |A|_{D_0,m}
 \\&\leq C_{m,s} t^{s-\ve} |A|_{D_0,s} + C_m' |A|_{D_0, s} |A|_{D_0,m} + C'_m |A|_{D_0,m} 
\\ &\leq C''_m |A|_{D_0,m}, \quad m >0 
\end{aligned}
\end{equation}   
where we used the initial condition \re{A_ic}. For the $s$ norm,  
we apply \re{I1_s_norm_est} and \re{I+K_inv_s_norm} to get 
\begin{equation} \label{I+K_inv_I1_s_norm} 
\begin{aligned} 
 |(I+\Kc)^{-1} I_1|_{D_0,s} 
&\leq C_s \left( |(I+\Kc)^{-1}|_{D_0, s} |I_1|_{D_0,\ve} + |(I+\Kc)^{-1}|_{D_0,\ve} |I_1|_{D_0,s} \right) 
\\ &\leq 2C_s |(I+\Kc)^{-1}|_{D_0,s} |I_1|_{D_0,s} 
\\ &\leq C_{r,s} t^{r-s} |A|_{D_0,r}, \quad 0 < s \leq r.  
\end{aligned} 
\end{equation} 
Using estimates \re{I2_s_norm_est}, \re{I2_m_norm_est}, \re{I+K_inv_m_norm} and \re{I+K_inv_s_norm} we get, 
\begin{equation} \label{I+K_inv_I2_m_norm}  
\begin{aligned}
 |(I+\Kc)^{-1} I_2|_{D_0,m}
 &\leq C_m (|(I+\Kc)^{-1}|_{D_0,m} |I_2|_{D_0,\ve} + |(I+\Kc)^{-1}|_{D_0,\ve} |I_2|_{D_0,m} ) 
 \\ &\leq C_m' (1+ t^{-\yh} |A|_{D_0,m} )  
 (t^{-\yh} |A|^2_{D_0,s}) + C_m' |A|_{D_0,m}
 \\ &\leq C_m' \left(  t^{-\yh} |A|^2_{D_0,s} 
+ t^{-1} |A|^2_{D_0,s} |A|_{D_0,m} 
 + |A|_{D_0,m} \right) 
 \\ &\leq 3C_m' |A|_{D_0,m}, \quad m >1, 
\end{aligned}
\end{equation} 
where we used the initial condition \re{A_ic}. 
For the $s$-norm, we apply \re{I2_s_norm_est} and \re{I+K_inv_s_norm} to get 
\begin{equation} \label{I+K_inv_I2_s_norm} 
   |(I+\Kc)^{-1} I_2|_{D_0,s} 
\leq C_s  |(I+\Kc)^{-1}|_{D_0,s} |I_2|_{D_0,s}
 \leq C_s' t^{-\yh} |A|_{D_0,s}^2, \quad s = 1 + \e_0.   
\end{equation} 

In a similar way, by using estimates \re{I3_s_norm_est}, \re{I3_m_norm_est}, \re{I4_s_norm_est} and \re{I4_m_norm_est}, we can show that 
\begin{equation} \label{I+K_inv_I3I4_m_norm}  
|(I+\Kc)^{-1} I_3|_{D_0,m}, \; |(I+\Kc)^{-1} I_4|_{D_0,m}
  \leq C_m |A|_{D_0,m}, \quad m > 1, 
\end{equation} 
and
\begin{equation} \label{I+K_inv_I3I4_s_norm} 
 | (I+ \Kc)^{-1} I_3|_{D_0,s} \leq C_{r,s} t^{r-s} |A|_r, \quad s \leq r; \qquad 
| (I+ \Kc)^{-1} I_4|_{D_0,s} \leq C_s t^{-\yh} |A|^2_s. 
\end{equation}
Combining estimates \re{I+K_inv_I1_m_norm}, \re{I+K_inv_I2_m_norm} and \re{I+K_inv_I3I4_m_norm}, 
we obtain for $\wti A = (I+\Kc)^{-1}(\sum_{j=1}^4 I_j)$ the following estimate for the $m$-norm:  
\begin{equation} \label{ti_A_m_est}   
 | \wti A|_{D_0,m} \leq C_m |A|_{D_0,m}, \quad m > 1. 
\end{equation} 
By using \re{I+K_inv_I1_s_norm}, \re{I+K_inv_I2_s_norm} and \re{I+K_inv_I3I4_s_norm}, we obtain the following estimate for the $s$-norm:  
\begin{align} \label{ti_A_s_est}   
  |\wti A|_{D_0,s} \leq C_{r,s} t^{r-s} |A|_{D_0,r} + C_s
  t^{-\yh} |A|^2_{D_0,s}, \quad s = 1 + \e_0, \quad s \leq r. 
\end{align} 
Finally, we estimate the norms of $A' = \wti A \circ G$, where $G=I+g =F^{-1}$. 
By \re{g_m_norm_est} and the initial condition \re{A_ic}, we have
\[
  |G|_{D_1,1+\ve} \leq C_1 |A|_{D_0, \yh} 
  \leq C_1 t^{\yh} < 1,   
\]
if we take $t< 1/C_1^2$.

Let $D_1 = F(D_0)$. Since $|f|_{D_0,1+\e_0}$ is small, we can assume that $D_1 \subset \subset \Uc_1 \subset \subset B_R$.
Applying the chain rule \re{chain_rule_a>1} and estimate \re{g_m_norm_est} for $G$, we obtain 
\begin{align*}
|A'|_{D_1,m} &= |\wti A \circ G|_{D_1,m}
\leq C_m (|\wti A|_{D_0,m} (1+|G|_{D_1, 1+ \frac{\e_0}{2}}^{\frac{1+\ve_0}{1+\ve_0/2 }} ) + | \wti A \|_{D_0,1+ \frac{\e_0}{2}} |G|_{D_1,m} + \| \wti A \|_{D_0,0} ) 
\\ &\leq C_m' ( |\wti  A|_{D_0,m} + |A|_{D_0, m-\yh} ), \quad m>1. 
\end{align*}
Here in the above line we applied \re{ti_A_s_est} with $r= 1+\e_0$ and $s = 1+ \e_0 /2$:  
\[
  |\wti A|_{D_0, 1+ \frac{\e_0}{2}} \leq C_s t^{\frac{\e_0}{2}} |A|_{D_0,1+\e_0} + C_s
  t^{-\yh} |A|^2_{D_0,1+ \frac{\e_0}{2}} \leq 2C_s.  
\]
Using estimates \re{ti_A_m_est} and \re{ti_A_s_est} for $\wti A$, we obtain 
\begin{equation} 
 \begin{gathered} 
|A'|_{D_1,s} \leq C_{r,s} t^{r-s} |A|_{D_0,r} + C_s t^{-\yh} |A|^2_{D_0,s}, \quad s = 1 + \e_0, \quad s \leq r; 
\\  
|A'|_{D_1,m} \leq C_m |A|_{D_0,m}, \quad m > 1. 
\end{gathered}  
\end{equation}

Notice that all the constants appearing in the above estimates are upper stable, in view of the remark after Definition \ref{Def::upp_low_stable}.   
We now summarize the estimates from this section in the following proposition. 
\begin{prop} \label{Prop::diffeo_est}  
  Let $D_0$ be a strictly pseudoconvex domain with $C^2$ boundary. Let $J$ be an almost complex structure defined on $\ov D_0$, given by the set of vector fields $\{ X_{\ov \all} \}_{\all=1}^n= \{ \pa_{\ov \all} + A^\beta_{\ov \all} \pa_\beta \}_{\all=1}^n$ (i.e. $S_J^+ = \Span 
 \{ X_{\ov \all} \}$ ). Let $F= I - \Ec_{D_0} S_t PA$ be given as above and set $F(D_0) = D_1$. Denote by $J'$ the push-forward of $J$ under $F$, such that $J'$ is given by the vector fields $ \{ X'_{\ov \all} \}_{\all=1}^n= \{ \pa_{\ov \all} + (A')^\beta_{\ov \all} \pa_\beta \}_{\all=1}^n$ on $D_1$. Let $s=1+\e_0$ and assume that
\[
  t^{-\yh} |A|_{D_0,s} < 1, \quad t < \frac{1}{(2C_1)^2}.  
\]
Then the following hold: 
\begin{enumerate}[(i)] 
\item $F$ is a diffeomorphism of $B(0,R)$ onto itself. The inverse is given by $G = I + g$, where $g$ satisfies the estimate:  
\[
  |g|_{B_R,m} \leq C_m |f|_{B_R,m} \leq C_m |A|_{D_0,m-\yh}, \quad m > 1. 
\]

\item \begin{equation} \label{A'_est}
 \begin{gathered} 
|A'|_{D_1,s} \leq C_{r,s} t^{r-s} |A|_{D_0,r} + C_s t^{-\yh} |A|^2_{D_0,s}, \quad s \leq r;  
\\  
|A'|_{D_1,m} \leq C_m |A|_{D_0,m}, \quad m > 1. 
\end{gathered}  
\end{equation} 
\end{enumerate} 
The constants $C_1, C_{r,s}, C_s, C_m$ are upper stable under small $C^2$ perturbation of the domain. 
\end{prop}  
\section{Iteration scheme and convergence of maps}  
In this section we set up the iteration scheme. We apply an infinite sequence of coordinate transformation $F_j$ as constructed in the previous section. The goal is to show that the composition of maps $\wti F_j := F_j \circ F_{j-1} \circ \cdots F_0$ converge to a limiting diffeomorphism $F$, while the perturbation $A_j$ converges to $0$. 

For this scheme to work we need to ensure that for each $i=0,1,2,\dots$, the map $F_i$ takes $D_i$ to a new domain $D_{i+1} = F(D_i)$ which is still strictly pseudoconvex with $C^2$ boundary. Hence we need to control the $C^2$-norm of the map $F_j$. 

In what follows we follow the same set-up as the last section and assume that 
\[
  D_0 = \{ z \in \Uc: \rho_0(z) < 0 \} \subset \subset \Uc \subset \subset B_0. 
\]
where $B_0$ is some large ball and $ \dist(D_j, \Uc)$ is bounded below by some positive constant. 
By applying extension, we assume that for each $j$, the map $F_j$ is a diffeomorphism from $B_0$ onto itself, $F_j$ is an identity map outside $\Uc$, and the defining function $\rho_j$ of the domain $D_j$ is defined in $\C^n$. 

We first recall two useful results from \cite{G-G24}. 
\begin{lemma}{\cite[Lemma 7.1]{G-G24}} \label{Lem::map_iter}   
Fix a positive integer $m$.  Let $D_0\subset \mathcal U\subset B_0 \subset \R^d$ with $\ov{D_0}\subset \mathcal U$.
Suppose that $D_0$ admits a $C^m$ defining function $\rho_0$ satisfying
$$
D_0=\{x\in \mathcal U\colon \rho_0(x)<0\}
$$
where $\rho_0>0$ on $\ov \Uc \setminus D_0$ and $\nabla \rho_0\neq0$ on $\pa D_0$.
Let
$F_j = I + f_j$ be a $C^m$ diffeomorphism which maps ${B_0}$ onto $B_0$ and maps $D_j$ onto $D_{j+1}$.
Let 
$\rho_1=(\widetilde E\rho_0)\circ F_0^{-1}$ and $\rho_{j+1}=\rho_j\circ F_j^{-1}$ for $j>0$, which are defined on $B_0$.
For any $\ve >0$, there exists
$$\si = \si(\rho_0, \ve, m)>0$$
such that if
\eq{}
\| f_j \|_{{B_0},m}\leq \frac{\si}{(j+1)^2},\quad 0\leq j< L,
\eeq
then the following hold
\bppp
\item
 $\tilde F_j=F_j\circ\cdots\circ F_0$ and $\rho_{j+1}$ satisfy
\ga\label{C2conv-g}
\| \tilde F_{j+1}-\tilde F_j \|_{{B_0},m}\leq C_m\frac{\si}{(j+1)^2},\quad 0\leq j< L\\
\| \tilde F_{j+1}^{-1}-\tilde F_j^{-1} \|_{{B_0},m}\leq C'_m\frac{\si}{(j+1)^2},\quad 0\leq j< L, \\
\| \rho_{j+1} - \rho_0 \|_{{\U},m} \leq \ve, \quad 0\leq j< L.
\label{Ucont}
\end{gather}
\item
 All $D_j$ are contained in $\mathcal U$ and
\eq{distDj}
\dist (\pa D_j,\pa D)\leq C \ve, \quad \dist(D_j,\pa {\mathcal U})\geq\dist(D_0,\pa {\mathcal U})-C \ve >0.
\eeq
\eppp
In particular, when $L=\infty$, $\widetilde F_j$ converges in $C^m$ to a $C^m$ diffeomorphism from ${B_0}$ onto itself, while $\rho_j$ converges in $C^m$  of ${B_0}$ as $\widetilde F^{-1}_j$ converges in $C^m$ norm on the set.
\end{lemma} 

\begin{lemma}{\cite[Lemma 7.2]{G-G24}} \label{Lem::Levi} 
   Let $D$ be a  relatively compact $C^2$ domain in $\mathcal U$ defined by a $C^2$ function $\rho$.
There are 
$\ve=\ve(\rho)>0$ and a neighborhood $\mathcal{N}=\mathcal{N}(\rho)$ of $\pa D$ such that if
 $ \| \ti \rho-\rho \|_{\Uc,2}<\ve$, then we have
 $$
\inf_{\tilde z, \tilde t, |\tilde t|=1}\{ L\ti \rho(\tilde z,\tilde t)\colon
 \ti t\in T_{F(z)}^{1,0}\tilde\rho, \tilde z\in \mathcal{N}\}
\geq \inf_{z, t,|t|=1} \{L\rho(z,t)\colon  t\in T_z^{1,0}\rho,  z\in \mathcal{N}\}-C
\ve.
 $$
  Furthermore, $\tilde D=\{z\in \mathcal U\colon\tilde\rho<0\}$ is a $C^2$ domain with
  $\pa \widetilde D\subset\mathcal N(\rho)$.  
\end{lemma}

Notice that for a bounded strictly pseudoconvex domain $D_0$ with $C^2$ defining function $\rho_0$, there is an $\ve(D_0)>0$ such that if $\| \rho - \rho_0 \|_2 < \ve(D_0)$, then all the constants in \rp{Prop::diffeo_est} can be chosen independent of $\rho$. Furthermore, by \rl{Lem::Levi}, the domain defined by $ \rho < 0$ is strictly pseudoconvex if $\ve(D_0)$ is sufficiently small.

Finally, we let 
\begin{equation} \label{del_rho_0}
  \si(\rho_0) = \si(\rho_0, \ve(D_0), 2))   
\end{equation}
be the constant from \rl{Lem::map_iter}. In particular if $\rho = \rho_0 \circ F^{-1}$, where $F = I + f$ and $\| f \|_{B_0,2} \leq \si$, then $\| \rho - \rho_0 \|_{\Uc,2} \leq \ve(D_0)$. We note that both $\ve(D_0)$ and $\si(\rho_0)$ are lower stable under small $C^2$ perturbation of the domain.  

\begin{prop} \label{Prop::iter}   
Let $r> 3/2$ and $s = 1+ \e_0 $ for some sufficiently small $\e_0>0$ (so that condition below are satisfied). Let $C_s, C_r$, $C_{r,s}, \ve(D), \si(\rho_0)$ be the constants stated above, and let $\all,\beta,d,\la, \gm$ be positive numbers satisfying  
\begin{gather} \label{iter_parameter}
  r-s - \la -\gm > \all d + \beta, \quad 
\all (2-d) > \yh + \la, \quad 
\beta (d-1) > \la, 
\quad 1 < d < 2. 
\end{gather} 
Note that the second and fourth conditions imply that $\all > 1/2$.  
 Let $D$ be a strictly pseudoconvex domain with a $C^2$ defining function $\rho_0$ on $\Uc$ and $ \{ X_{\all} \}_{\all=1}^n = \{ \pa_{\ov \all} + A_{\ov \all}^\beta \pa_{\beta} \}_{\all=1}^n \in \La^r(\ov D)$ be a formally integrable almost complex structure. There exists a constant 
\[
  \hat t_0 = \hat t_0 (s, r, \all, \beta, d, \la, C_{r,s}, C_s, C_r, \si(\rho_0), \ve(D), |A|_{D,r})   
\] 
such that if 
\[
  |A|_{D,s} \leq t_0^{\all}, \quad 0 < t_0 \leq \hat t_0, 
\]
then the following statements are true for $i=0,1,2,\dots$
\begin{enumerate}[(i)] 
\item 
There exists a diffeomorphism $F_i = I + f_i$ from $B_0$ onto itself with $F_i^{-1} = I + g_i$ such that 
 $f_i, g_i$ satisfy 
\[
|g_i|_{B_0, m} \leq |f_i|_{B_0,m}, \quad m >1.     
\]
\item Set $\rho_{i+1} = \rho_i \circ F_i^{-1}$, and denote $D= D_0$. Then 
$D_{i+1}:= F_i(D_i) = \{ z \in \Uc: \rho_{i+1} <0\} $ and
\begin{gather*} 
  \| \rho_{i+1} - \rho_0 \|_{\Uc,2} \leq \ve (D_0), 
 \\ 
 \dist(D_{i+1}, \pa \U) \geq \dist (D_0, \pa \U) - C\ve > 0. 
 \end{gather*} 
\item 
For $i \in \N$, we have $\Span 
\{\pa_{\ov \all} + A_{i+1} \pa_{\all} \} = dF_i|_{\ov D_i} \left[ \Span 
\{ \pa_{\ov \all} + A_i \pa_{\all} \} \right] $. Moreover 
\[
  |A_i|_{D_i,s} \leq t_i^\all, \quad 
  |A_i|_{D_i,r} \leq |A|_{D,r}  t_i^{-\beta}.  
\] 
\end{enumerate}
The constant $\hat t_0$ needs to converge to $0$ as $r \to \frac{3}{2}^+$, and $\hat t_0$ is lower stable under small $C^2$ perturbation of the domain. 
\end{prop}
\begin{proof} 
We prove by induction on $i$. First we prove (i)-(iii) for $i=0$. We will write $A_0 = A$ and $D_0 = D$. 
Fix $s=1+\e_0$, $r > 3/2 $, and set $a_i := |A_i|_{D_i,s}$, $L_i := |A_i|_{D_i,r}$. 
Choose  
\begin{equation} \label{t0hat_Cs_bd}  
  \hat t_0 \leq \frac{1}{(2C_1)^2} < 1, 
\end{equation}
where $C_1$ is given by \rp{Prop::diffeo_est}. In particular we also have 
\[ 
 t_0^{-\yh} |A|_{D_0,s} \leq t_0^{\all -\yh} < 1, \quad 0 < t_0 < \hat t_0.
\]  
Thus the hypothesis of \rp{Prop::diffeo_est} are satisfied for $0 < t_0 < \hat t_0$. 
On $D_0$ we have the homotopy formula $A_0 
= \db P_0 A_0 + Q_0 \db A_0 $.  
Set $F_0 = I + f_0$, where $f_0 = - \Ec_{D_0} S_{t_0} P_0 A_0$ and $\Ec_{D_0}$ is the Rychkov extension operator on $D_0$.  
By \rp{Prop::diffeo_est}, $F_0$ is a diffeomorphism of $B(0,R)$ onto itself, with inverse $F_0^{-1} = I + g_0$, and $g_0$ satisfies the estimate:  
\[
  |g_0|_{B_R,m} \leq C_m |f_0|_{B_R,m}, \quad m > 1. 
\]
This proves part (i) for the case $i=0$. 

Next we verify part (ii) when $i=0$. 
By \re{f_m_norm_est_2} we have 
\begin{equation} \label{f0_C2_norm}      
  \| f_0 \|_{B_0,2 } \leq |f_0 |_{B_0, 2+\e_0} \leq C_2  t_0^{-\yh} |A_0|_{D_0, 1+\e_0}  
  \leq C_2 t_0^{\all -\yh}. 
\end{equation} 
Let $\si = \si(\rho_0)$ be the constant in \re{del_rho_0}, and assume that $\hat t_0$ further satisfies   
\begin{equation} \label{t0hat_C2_cond_i=0} 
  \hht t_0 \leq \left( \frac{\si}{C_2} \right)^{\frac{1}{\all-\yh}}, \quad \all > \yh. 
\end{equation}
Then \re{f0_C2_norm} and \re{t0hat_C2_cond_i=0} together imply that
$\| f_0 \|_{B_0,2} \leq \si$ for $0 < t_0 < \hat t_0$. Set $\rho_1 = \rho_0 \circ F_0^{-1}$ and $D_1 = F_0(D_0) = \{ z \in \Uc: \rho_1 < 0 \}$. 
By \rl{Lem::map_iter}, we get   
\begin{gather*} 
  \| \rho_1 - \rho_0 \|_{\Uc,2} \leq \ve(D_0), \qquad 
\dist(D_1, \pa \Uc) \geq \dist (D_0, \pa \Uc) - C \ve. 
\end{gather*}
This proves (ii) for $i=0$. We note that both $\ve(D_0)$ and $\si(D_0)$ are lower stable.   

We now verify (iii) when $i=0$. On $D_1$, let $A_1$ be the coefficient of the new structure obtained by the push-forward of $F_1$, i.e. $\Span 
\{\pa_{\ov \all} + A_1 \pa_{\all} \} = dF_i|_{\ov D_i} \left[ \Span 
\{ \pa_{\ov \all} + A_0 \pa_{\all} \} \right] $. 
By \rp{Prop::diffeo_est} we have 
\begin{equation} \label{a1_L1_est} 
\begin{gathered}
  a_1 \leq C_{r,s} L_0 t_0^{r-s}  + C_s t_0^{-\yh} a_0^2 \leq C_{r,s} L_0 t_0^{r-s} + C_s 
t_0^{2\all-\yh}; 
  \\ 
  L_1 \leq C_r L_0.   
\end{gathered} 
\end{equation} 
For some fixed $\la >0$, we require the additional assumption on $\hat t_0$: 
\begin{equation} \label{t0hat_min}   
  \hat t_0 \leq \min \left\{ \left( \frac{1}{2C_{r,s}}  \right)^{\frac{1}{\la}}, \left( \frac{1}{2C_r}  \right)^{\frac{1}{\la}}, \left( \frac{1}{2C_s}  \right)^{\frac{1}{\la}} \right\}.  
\end{equation} 
Then for all $0 < t_0 < \hat t_0$, we have $C_{r,s}, C_s, C_r \leq 
\yh t_0^{-\la}$. For $\gm >0$, we further require  
\begin{equation} \label{t0hat_L0_bd} 
  \hat t_0 \leq \left( \frac{1}{L_0} \right)^{\frac{1}{\gm}}, \quad L_0:= |A_0|_{D_0,r},
\end{equation} 
so that $L_0 \leq t_0^{-\gm}$ for $0 < t_0 < \hat t_0$. Hence we get from \re{a1_L1_est} that 
\begin{gather*}
  a_1 \leq \yh ( t_0^{r-s-\gm-\la} + t_0^{2\all-\yh-\la} ) \leq t_0^{d \all} = t_1^\all ; 
  \\ 
  L_1 \leq t_0^{-\la} L_0 \leq t_0^{-\beta d} L_0 = t_1^{-\beta} L_0, 
\end{gather*}
where we have assumed the following constraints:
\begin{equation} \label{range_initial}  
\begin{gathered} 
\all d  < r-s - \gm -\la
\\ 
\all (2-d) > \yh  + \la, 
\qquad 
 \beta d > \la.  
\end{gathered}
\end{equation} 
Thus we have verified (iii) for $i=0$ assuming the intersection of the above constraints is nonempty. We will see in the induction step that this is true provided $r > s + \yh$. 

Now assume that (i) - (iii) hold for some $i-1 \in \N$. We shall verify the induction step. Let 
$t_i = t_{i-1}^d$, where $d \in (1,2)$ is to be specified. Suppose we have found $D_i = F_{i-1} 
(D_{i-1})$ which is still strictly pseudoconvex with $C^2$ boundary. Apply the homotopy formula to get $A_i = \db P_i A_i + Q_i \db A_i$ on $D_i$.
Let $F_i = I + f_i$, where $f_i = - \Ec_{D_i} S_{t_i} P_i A_i$, and $\Ec_{D_i}$ is the Rychkov extension operator on $D_i$. 
Note that we still have $t_i < \frac{1}{(2C_1)^2} <1$ and 
\begin{equation} \label{A_s_norm_ind_ic}   
   t_i^{-\yh} |A_i|_{D_i,s} \leq t_i^{\all-\yh} 
   <1 . 
\end{equation}
Hence we can apply \rp{Prop::diffeo_est} (i) to show that and $F_i$ is a diffeomorphism on $B_0$ and the inverse $I + g_i$ satisfies the estimate 
\[
  |g_i|_{B_0, m} \leq |f_i|_{B_0,m}, \quad m >1. 
\]
This verifies the induction step for part (i).  

Define 
\[
  D_{i+1} = \{ z \in \Uc: \rho_{i+1} (z) <0 \}, 
\]
where $\rho_{i+1} (z) = \rho_i \circ F_i^{-1}$. By \re{f_m_norm_est} and the induction hypothesis for (iii), we have 
\begin{equation} \label{fi_C2_induc} 
  \| f_i \|_{B_0,2} \leq |f_i |_{B_0, 2+\ve'} 
  \leq C_2 t_i^{-\yh} |A_i|_{D_i, 1+\ve'} 
  \leq C_2 t_i^{\all-\yh} = C_2 t_0^{(\all-\yh)d^i},
\end{equation}
where we choose $0<\ve'<\e_0$.  
Now we require that  
\begin{equation} \label{t0hat_C2_norm_induc}  
  C_2 \hat t_0^{(\all-\yh) d^i} \leq \frac{\si}{(i+1)^2}.    
\end{equation}
This has been achieved for $i=0$ by \re{t0hat_C2_cond_i=0}. Suppose \re{t0hat_C2_norm_induc} holds for $i-1$. Then 
\[
  C_2 \hat t_0^{(\all-\yh) d^i} 
  = C_2 \hat t_0^{(\all-\yh) d^{i-1}} \hat t_0^{(\all-\yh) d^{i-1}(d-1)} \leq \frac{\si}{i^2} \hat t_0^{(\all-\yh) d^{i-1}(d-1)} \leq \frac{\si}{(i+1)^2},   
\]
where the last inequality holds for all $i \geq 1$. 
Therefore by \re{fi_C2_induc} and \re{t0hat_C2_norm_induc}, we have $\| f_i \|_{B_0,2} \leq \frac{\si}{(i+1)^2}$. 
It then follows from \rp{Lem::map_iter} (i) that 
\begin{equation} \label{rho_i_induc_est}  
    \| \rho_{i+1} - \rho_0 \|_{\Uc, 2} \leq \ve(D_0), \qquad \dist(D_{i+1}, \pa \U) \geq \dist (D_0, \pa \U) - C\ve.  
\end{equation}
This shows that $D_{i+1}$ is still a strictly pseudoconvex domain with $C^2$ boundary, and we have verified the induction step for part (ii). In addition, \re{rho_i_induc_est} with our choice of 
$\ve(D_0)$ allows us to apply \rp{Prop::diffeo_est} with all the constants independent of $i \in \N$. 

Next, we verify the induction step for (iii). On $D_{i+1}$, let $A_{i+1}$ be the coefficient matrix  such that 
$\{\pa_{\ov \all} + A_{i+1} \pa_{\all} \} = dF_i|_{\ov D_i} \left[ \Span 
\{ \pa_{\ov \all} + A_i \pa_{\all} \} \right] $.
Apply \rp{Prop::diffeo_est} to get: 
\begin{gather*}
  a_{i+1} \leq C_{r,s} t_i^{r-s} L_i + C_s t_i^{-\yh} a_i^2 
  \leq C_{r,s} L_0 t_i^{r-s -\beta} + C_s t_i^{2 \all-\yh} ;  
  \\ 
  L_{i+1} \leq C_r L_i \leq C_r L_0 t_i^{-\beta},   
\end{gather*}
where we used the induction hypothesis $a_i \leq t_i^\all$ and $L_i \leq L_0 t_i^{-\beta}$. 
Notice that by the condition \re{t0hat_min}, we still have $C_{r,s}, C_s, C_r \leq \yh t_i^{-\la}$ since $t_i < t_0$. Similarly condition \re{t0hat_L0_bd} implies that $L_0 \leq t_i^{-\gm}$. Hence
\begin{gather*}
  a_{i+1}
\leq \yh ( t_i^{r-s} t_i^{-\la - \gm -\beta} 
  + t_i^{-\la} t_i^{-\yh} t_i^{2\all} ) 
  \leq t_i^{d \all} = t_{i+1}^\all, 
  \\ 
  L_{i+1} \leq t_i^{-\la} L_i \leq t_i^{-\la - \beta} L_0 
  \leq t_i^{-d\beta}L_0 = t_{i+1}^{-\beta} L_0, 
\end{gather*} 
 where we have assumed 
\begin{equation} \label{range_induc}    
 \begin{gathered}
 \all d + \beta < r -s - \la - \gm, 
              \\ 
 \all (2-d) > \yh + \la, 
\qquad 
 \beta > \frac{\la}{d-1}. 
\end{gathered} 
\end{equation}
Notice that the above constraint is more strict than \re{range_initial}. 
Let $\Dc(r,s, d)$ be the set of $(\all,\beta, \gm, \la)$ such that \re{range_induc} is satisfied. We now determine the values of $r,s, d$ such that $\Dc(r,s,d)$ is non-empty. Consider the limiting domain of $\all$ for fixed $r,s, d$ and $\beta, \la, \gm =0$: 
\[
  \Dc_\ast(r,s, d) = \left\{ \all \in (0,\infty): \all d < r-s, \quad \all(2-d) > \yh, \quad \all > \yh \right\}.  
\]  
Hence $\Dc_\ast(r,s,d)$ is non-empty if and only if 
\[
  r-s > p(d), \quad p(d) := \frac{d}{2(2-d)}. 
\]
On the interval $(1,2)$, $p$ is a strictly increasing function with infimum value $p(1)= \yh$. This implies that
\begin{equation} \label{rs_range}
  r-s > p(1) = \yh, \quad r > s + \yh > \frac{3}{2}  \quad (\text{since $s>1$}).  
\end{equation} 
Notice that under the above condition for $r,s$, $\Dc(r,s, d)$ is still non-empty for sufficiently small $\beta, \la, \gm$. In summary, given $r= \frac{3}{2} + \ti \e_0$, we first choose $s=1+\e_0$ for $0 < \e_0< \ti \e_0$, such that \re{rs_range} is satisfied. This is possible by choosing $d \in (1,2)$ sufficiently close to $1$. We then choose $(\all,\beta,\gm,\la) \in \Dc(r,s, d)$ with $\all $ close to $1/2$, and $\beta, \la,\gm$ sufficiently small positive number, such that \re{range_induc} holds. In view of \re{t0hat_C2_cond_i=0}, \re{t0hat_min}, \re{t0hat_L0_bd} and \re{t0hat_C2_norm_induc}, $\hat t_0$ needs to be chosen sufficiently small. In other words, $\hat t_0 \to 0$ as $r \to \frac{3}{2}^{+}$. On the other hand, we observe that the constants $C_1, C_2, C_s, C_r, C_{r,s}$ showing up in \re{t0hat_Cs_bd}, \re{t0hat_C2_cond_i=0}, \re{t0hat_min} and \re{t0hat_C2_norm_induc} are upper stable, and the constants $\si$ in \re{t0hat_C2_norm_induc} is lower stable, we conclude that $\hat t_0$ is lower stable under small $C^2$ perturbation of the domain, once we fix $r>\frac32$ and $|A|_{D_0,r}$. 
\end{proof} 

The following consequence of the above result is what is actually used in the proof of \rt{Thm::intro_main_thm}. 
\begin{cor} \label{Cor::A_m_norm_est}  
   Let $s=1+\e_0$ and $r_0 = \frac32 + \ti \e_0$, where $\e_0,\ti \e_0$ are small positive constants satisfying $0 < \e_0 < 
\ti \e_0$. Let $D$ be a strictly pseudoconvex domain with a $C^2$ defining function $\rho_0$ on $\Uc$ and $ \{ X_{\ov \all} \}_{\all=1}^n = \{ \pa_{\ov \all} + A_{\ov \all}^\beta \pa_{\beta} \}_{\all=1}^n \in \La^{r_0}(\ov D)$ be a formally integrable almost complex structure. For each $i \in \N$, let $F_i= I +f_i, D_i, A_i$ be given as in the proof of \rp{Prop::iter}.  
There exist $\del_0= \del_0(\ti \e_0, |A|_{D, r_0}, \si(\rho_0), \ve(D))$, $1<d<2$, $\all > \yh$, $\eta >0$, and a constant $N = N(m,d)$ such that if $|A|_{D,s} < \del_0$, then 
\begin{gather*}
  |A_i|_{D_i,s} \leq t_i^\all, 
 \quad \text{for all $i \in \N$}; 
\\ 
|A_i|_{D_i,m}\leq |A_N|_{D_N,m} t_i^{-\eta}, \quad \text{for all $i > N $}.    
\end{gather*} 
Here $t_{i+1} =t_i^d$ and $\eta$ is independent of $m$. The constant $\del_0$ needs to converge to $0$ as $\ti \e_0 \to 0$, and $\del_0$ is lower stable under small $C^2$ perturbation of the domain. 
\end{cor}  
\begin{proof}
Denote $A = A_0$. We apply \rp{Prop::iter} with $r =r_0 = \frac32 + \ti \e_0$ and choose the parameters $\all,\beta, d, \la$ satisfying the conditions in \re{iter_parameter}. Then there exists a constant $\del_0= \del_0(\ti \e_0, |A|_{D,r_0}), \si(\rho_0), \ve(D)) $, and a sequence $\{ t_i \}_{i \in \N}$ such that if $|A|_{1+\e_0} < \del_0 = t_0^\all$, then for all $i \in \N$, we have $|A_i|_{D_i,s} 
\leq t_i^\all$ with $t_{i+1} = 
t_i^d$, $\all > \yh$, $1<d<2$. The constant 
$\del_0$ tends to 0 as $\ti \e_0$ and it is lower stable under small $C^2$ perturbation of the domain. Denote $M_i:= |A_i|_{D_i,m}$, for $m > 1$. 
By \rp{Prop::diffeo_est}, we have the estimate \re{A'_est} 
\[
   M_{i+1} \leq C_m M_i, \quad i \in \N. 
\]
Fix $\la >0$. One can find a large $N = N(m,d) \in \N$ such that
\begin{equation} \label{Cm_bd_i>N}
  C_m \leq t_i^{-\la}, \quad \text{for all $i \geq N$.} 
\end{equation}
We would like to show that there exists $\eta>0$, independent of $m$, such that 
for all $i \geq N$, the following holds 
\[
  M_i \leq M_N t_i^{-\eta}, \quad i \geq N. 
\] 
For $i =N$, the above inequality is obvious. Assume it holds for some $i \geq N$.  
Then 
\[
  M_{i+1} \leq C_m M_i \leq t_i^{-\la -\eta} M_N
  \leq t_i^{-d \eta} M_N = t_{i+1}^{-\eta} M_N, \quad i \geq N, 
\] 
where the third inequality in the line above holds if we choose 
\[
  \eta(d-1) > \la.  \qedhere
\] 
\end{proof} 
We are now ready to prove \rt{Thm::intro_main_thm}. 
\begin{prop} \label{Prop::main}  
  Let $D$ be a strictly pseudoconvex domain with $C^2$ boundary in $\C^n$ and let $\{ X_{\ov \all} = \pa_{\ov \all} + A^\beta_{\ov \all} \pa_\beta \}_{\all=1}^n$ be a formally integrable almost complex structure on $D$ of the class $ \La^{\frac32+\ti \e_0}(\ov D)$, for any small $\ti \e_0>0$. Fix $\e_0$ with $0< \e_0 < \ti \e_0$. There exists an 
\[
\del_0= \del_0\left( \ti \e_0,|A|_{D,\frac32+ \ti \e_0}, \si(\rho_0), \ve(D) \right) 
\]
such that if $|A|_{D,1+\e_0} \leq \del_0$,  
then the following statements are true. 
  \begin{enumerate}[(i)]
  \item
  Let $A \in \La^m(\ov D)$, with $m > \frac{3}{2} + \ti \e_0$. There exists a $C^1$ diffeomorphism $F: \ov D \to \C^n $ such that if $A \in \La^m(\ov D)$, $m > \frac{3}{2} + \ti \e_0$, then $F \in \La^{m+\yh^{-}}(\ov {D})$.     
  \item 
    If $A \in C^\infty(\ov D)$, then $F \in C^\infty(\ov D)$. 
  \item 
  $dF (X_{\ov \all})$ are in the span of $
  \{\pa_{\ov 1}, \dots, \pa_{\ov n} \}$ and $F(\ov{D})$ is strictly pseudoconvex. 
  \end{enumerate}
The constant $\del_0$ needs to converge to $0$ as $\ti \e_0 \to 0$, but is independent of $m$ away from $3/2$. Furthermore, $\del_0$ is lower stable under small $C^2$ perturbations of the domain. 
\end{prop}
\begin{proof}
We will write $s=1+\e_0$ and denote $A_0 = A$ and $D_0 = D$. For each $i \in \N$, let $F_i=I+f_i, D_i, A_i$ be given as in the proof of \rp{Prop::iter}.  

Write $l = (1-\thh)s + \thh m$, where $s < l < m$ and $\thh \in (0,1)$ is to be chosen. By \rc{Cor::A_m_norm_est}, there exist $ \del_0= \del_0( \ti \e_0,|A|_{D,\frac32+ \ti \e_0}, \si(\rho_0), \ve(D) ), d \in (1,2)$, $\all >\yh$, $\eta >0$, $ \{t_i\}_{i=0}^\infty$ such that if $|A|_{D,1+\e_0} < \del_0= t_0^\all$, then 
\begin{equation} \label{Ai_sm_norm_est}  
\begin{gathered} 
   |A_i|_{D_i,1+\e_0} \leq t_i^{\all}, \quad i \in \N; 
   \\ 
   |A_i|_{D_i,m} \leq  t_i^{-\eta} M_N , \quad t_{i+1} = t_i^d, \quad i \geq N = N(m,d),  
\end{gathered} 
\end{equation}
where we denote $M_N:= |A_N|_{D_N,m} $. 
Here $\eta$ satisfies the condition
\begin{equation} \label{yh_est_eta_lb} 
   \eta > \frac{\la}{d-1}, 
\end{equation}
where $\la>0$ is some fixed constant for which 
\begin{equation} \label{yh_est_Cm_bd} 
  C_m \leq t_i^{-\la}, \quad i \geq N=N(m,d). 
\end{equation} 
Here we point out the crucial fact that the constant $\del_0$ in the smallness assumption of 
$|A|_{D,1+\e_0}$ does not depend on $\eta$ and $\la$, since \re{yh_est_Cm_bd} can always be satisfied for any $\la>0$ by choosing $N$ sufficiently large without making $t_0$ small. 

Using convexity of H\"older-Zygmund norm \re{H-Z_convexity} and estimate \re{f_m_norm_est}, we have for all $i \geq N$, 
\begin{equation} \label{fi_l+yh_norm}     
   |f_i|_{D_i, \ell+\yh} \leq C_{m,s} |f_i|^{1-\thh}_{D_i,s+\yh} |f_i|^\thh_{D_i,m+\yh}
\leq C'_{m,s} |A_i|_{D_i,s}^{1-\thh} |A_i|_{D_i,m}^\thh 
\leq C'_{m,s} M_N t_i^{(1-\thh)\all-\thh \eta}.    
\end{equation}
 
Consider the composition map $\wti F_j = F_j \circ F_{j-1} \circ \cdots \circ F_0$, where $F_j = I + f_j$ for $j \geq 0$. By using \rl{Lem::comp_long} and above estimate for $f_j$, we obtain for all $j \geq N$,   
\begin{equation} \label{Fj_ti_diff} 
\begin{aligned}
  | \wti F_j - \wti F_{j-1} |_{D_0, \ell+\yh}
  &= | f_j \circ F_{j-1} \circ \cdots F_0 |_{D_0, \ell+\yh} 
  \\ &\leq C_\ell^j \left\{ | f_j |_{\ell+\yh} + \sum_{0 \leq i\leq j} \left( | f_j |_{1+\e_0} | f_i |_{\ell+\yh} + C_{1/\e_0} | f_j |_{\ell+\yh} | f_i |_{1+\e_0} \right) \right\}. 
\end{aligned}   
\end{equation}
Set $\mu:=(1-\thh)\all- \thh \eta$. By choosing $0<\thh<\frac{\all}{\all + \eta} < 1$, we have $\mu>0$. 
For the first sum in the last line, we have  
\begin{align*}
  \sum_{0 \leq i \leq j} |f_j|_{1+\e_0} |f_i|_{\ell +\yh} 
&\leq |A_j|_{D_j,\yh+\e_0}  C_{m,s} M_N \sum_{0 \leq i \leq j} t_i^{(1-\thh)\all-\thh \eta}  
\leq C_{m,s} M_N t_j^{\all}. 
\end{align*} 
And the second sum in \re{Fj_ti_diff} is bounded by 
\[
 \sum_{0 \leq i \leq j} |f_j|_{\ell+\yh} |f_i|_{1+ \e_0} \leq C_{m,s} M_N t_j^{(1-\thh)\all - \thh \eta} \sum_{0 \leq i \leq j} t_i^{\all} 
 \leq C'_{m,s} M_N t_j^\mu. 
\] 
It follows from \re{fi_l+yh_norm} and \re{Fj_ti_diff} that for all $j \geq N$,    
\begin{equation} \label{ti_Fi_diff}
\begin{aligned} 
  | \wti F_j - \wti F_{j-1} |_{D_0, \ell+\yh}
&\leq C_\ell^j C_{m,s} M_N t_j^\mu 
=C_{m,s} M_N C_\ell^j  t_0^{d^j \mu},  
\end{aligned} 
\end{equation} 
 As a consequence, $\wti F_j$ is a Cauchy sequence in $\La^{\ell+\yh}(\ov D_0)$ when $j$ is sufficiently large. This can be seen by writing
\begin{align*}
   C_\ell^j t_0^{d^j \mu} 
&= C_\ell^j (t_0^\mu)^{d^{\frac{j}{2}} d^{\frac{j}{2}} }  = C_\ell^j \left[ (t_0^\mu)^{d^{\frac{j}{2}}} \right]^{d^{\frac{j}{2}}} \leq C_\ell^j \left[ (t_0^\mu)^{d\frac{j}{2}} \right]^{d\frac{j}{2}}  
\\ & = ( C_\ell^2)^{\frac{j}{2}}  \left[ \left( t_0^{\mu d}\right)^{\frac{j}{2}} \right]^{d\frac{j}{2}} 
\leq \left[ C_\ell^2 \left(t_0^{\frac{\mu d
}{2}}\right)^j \right]^{d\frac{j}{2}} . 
\end{align*}
For fixed $t_0, d, \mu$, by choosing $j> N(\ell,d)$, we can make the expression inside the bracket less than $1$, and thus $\wti F_j$ is a Cauchy sequence. 

We now make a summary and indicate the way the parameters are chosen. The constants $\all, d$ are first chosen to apply \rc{Cor::A_m_norm_est} and obtain estimate \re{Ai_sm_norm_est}. 
For any fixed $\ve>0$, in view of \re{yh_est_eta_lb}, we can choose $\thh_\ast = \thh_\ast(\ve)$ close to $1$, such that $l = s+\thh_\ast (m-s) > m-\ve$. We then choose $\eta_\ast = \eta_\ast (\ve)$ small such that $\frac{\all}{\all+\eta_\ast} > \thh_\ast$, which makes $\mu = (1-\thh_\ast) \all - \thh_\ast \eta_\ast$ positive. 

By \re{yh_est_eta_lb} this in turns requires that we choose $\la_\ast = \la_\ast(\ve)$ small so that $\eta_\ast > \frac{\la_\ast}{d-1}$, while $d$ has been fixed. Finally we choose $N = N(m,d,\ve)$ to be large such that \re{yh_est_Cm_bd} is satisfied. This shows that $\wti F_j$ is a Cauchy sequence in $\La^{m+\yh-\ve}(\ov{D_0})$ if $j > N(m,d,\ve)$. 
In other words, we have shown that if $A \in \La^m(\ov D)$ and $|A|_{1+\e_0} \leq \del_0$, then $\wti F_j$ converges to some limit $F \in \La^{m+\yh-\ve}(\ov{D})$ for any $\ve>0$, and $\del_0$ is independent of $\ve$.  
\\[10pt]  
(ii) By assumption we have $A \in \La^m(\ov D_0)$ for any $m>1$. 
The same argument as in (i) shows that $\{\wti F_j\}$ is a Cauchy sequence in $\La^{m + \yh-\ve}(\ov{D_0})$ for $j > N(m,d,\ve)$.   
Since this holds for all $m$, we have $F \in C^\infty(\ov{D_0})$. 

We now show that $F = \lim_{j \to \infty} \wti F_j$ is a diffeomorphism on $B_0$. By the inverse function theorem, it suffices to check that the Jacobian of $F(x)$ is invertible at every $x \in B_0$.  
Write $DF= I - (I-DF)$, then $DF$ is invertible with 
inverse $(I-(I-DF))^{-1}$ if $\|DF-I\|_{B_0,0} < 1$. Write 
\begin{align*}
 DF-I &= \lim_{j \to \infty} D \wti F_{j}  - I
\\ &= \lim_{j \to \infty} [D \wti F_{j} - D \wti F_{j-1} ]
+ [D \wti F_{j-1} - D \wti F_{j-2} ] + \cdots [D \wti F_1 - D \wti F_0] 
+ [D \wti F_0 - D \wti F_{-1}] ,   
\end{align*}
where we set $\wti F_{-1}$ to be the identity map and thus $D \wti F_{-1} = I $. 
Then 
\[
  \| DF- I \|_{B_0,0} \leq \sum_{j=0}^\infty 
  \|D \wti F_{j} - D \wti F_{j-1} \|_{B_0,0} \leq \sum_{j=0}^\infty \| \wti F_j - \wti F_{j-1} \|_{B_0,1}.
\]
Using \re{comp_est_Holder}, we get 
\begin{align*}
\| \wti F_j - \wti F_{j-1} \|_{B_0, 1} 
  &= \| f_j \circ F_{j-1} \circ \cdots F_0 
  \|_{B_0, 1}
  \\ &\leq C_1^j \left( \| f_j \|_{B_0,1} +  \| f_j \|_{B_0,1} \sum_{0 \leq i \leq j -1} \| f_i \|_{B_0,1} \right) \\ &\leq C_1^j \| f_j \|_{B_0,1} \left( 1+ \sum_{0 \leq i \leq j-1} t_i^\all \right) 
\leq C_1^j t_j^\all,  
\end{align*}
where we use the estimate $\| f_j \|_{B_0,1} = 
\| \Ec_{D_j} S_{t_j} P_j A_j \|_{B_0,1} \leq C_1 |A_j|_{D_j,1+\e_0} \leq C_1 t_j^\all$. 

Hence we have
\begin{align*}
  \| DF- I \|_{B_0,0} 
&\leq \sum_{j=1}^\infty \| \wti F_j - \wti F_{j-1} \|_{B_0,1}
\leq \sum_{j=1}^\infty C_1^j t_j^\all 
 = \sum_{j=1}^\infty C_1^j t_0^{d^j \all} 
\end{align*} 
The last expression converges to a number less than $1$ if we choose $t_0$ to be smaller than a constant depending on $d, \all, C_1$. 
\\[10pt]  
(iii) 
By \rl{Lem::A_choc}, $d \wti F_j (X_{\ov \all})= dF_j \cdots dF_0 (X_{\ov \all}) \in \Span \{\pa_{\ov \all} + (A_{j+1})_{\ov \all}^\beta \pa_\beta \} $. Since $dF(X_{\ov \all}) = \lim_{j\to \infty} d \wti F_j (X_{\ov \all}) $ and $\lim_{j \to \infty} |A_j|_{D_j, s} \leq \lim_{j\to \infty} t_j^\all = 0$, we have $d \wti F (X_{\ov \all}) = \Span \{\pa_{\ov \all} \}$ on $F(D_0)$. Let $\rho_0$ be the defining function for $D_0$, and for $j \geq 0$, let $\rho_{j+1} = \rho_j \circ F_j^{-1}$ be the defining function for $D_{j+1} = F_j (D_j)$. We have shown in the proof of 
\rp{Prop::iter} that 
\[
  \| f_j \|_{B_0,2} \leq \frac{\si(\rho_0)}{(j+1)^2}, \quad j=0,1,2,\dots. 
\]
Therefore by \rl{Lem::map_iter}, we have $\| \rho_j - \rho_0 \|_{\Uc,2 } \leq \ve(D_0) $ for all $j \in \N$ and $\rho_j$ converges in the $C^2$ norm on $B_0$ to $\rho$, with $F(D_0) = \{ z \in \Uc, \rho(z)<0 \}$. By our choice of $\ve(D_0)$ (see the remark after \rl{Lem::Levi}), $F(D_0)$ is strictly pseudoconvex. 
\end{proof}

\bigskip 
\bibliographystyle{amsalpha}
\bibliography{reference} 

\end{document}